\documentclass[a4paper, twoside]{article}
\usepackage{a4}

\usepackage{amsfonts}
\usepackage{amsmath}
\usepackage{amssymb}
\usepackage{mathrsfs}
\usepackage{vmargin}

\setmarginsrb{20mm}{20mm}{20mm}{20mm}{10mm}{10mm}{10mm}{10mm}

\def\barf{{\overline {f}}}
\def\barg{{\overline {g}}}
\def\fb{{f_{\overline{z}}}}
\def\gb{{g_{\overline{z}}}}
\def\div{{\rm div}}
\def\Re{\mathfrak{Re}}

\newcommand{\p}{{$p\mspace{1mu}$}}

\RequirePackage{amsthm}

\begin{document}
\long\def\symbolfootnote[#1]#2{\begingroup%
\def\thefootnote{\fnsymbol{footnote}}\footnote[#1]{#2}\endgroup}

\title {The geometry of planar \p-harmonic mappings: convexity,  level curves and the isoperimetric inequality
}
\author{ Tomasz Adamowicz
}
\date{}
\maketitle

\theoremstyle{plain}
\newtheorem{definition}{Definition}
\newtheorem{theorem}{Theorem}
\newtheorem{lem}{Lemma}
\newtheorem{cor}{Corollary}
\newtheorem{ex}{Example}
\newtheorem{observ}{Observation}
\newtheorem{opprb}{\bf Open problem}

\theoremstyle{definition}
\newtheorem{rem}{Remark}

\newcommand{\Om}{\Omega}
\newcommand{\R}{\mathbb{R}}
\newcommand{\kom}[1]{}
\renewcommand{\kom}[1]{{\bf [#1]}}

\begin{abstract}
 We discuss various representations of
planar \p-harmonic systems of equations and their solutions. For
coordinate functions of \p-harmonic maps we analyze signs of their
Hessians, the Gauss curvature of \p-harmonic surfaces, the length of
level curves as well as we discuss curves of steepest descent. The
isoperimetric inequality for the level curves of coordinate functions of planar \p-harmonic maps is proven. Our main
techniques involve relations between quasiregular maps and planar
PDEs. We generalize some results due to P.~Lindqvist,
G.~Alessandrini, G.~Talenti and P.~Laurence.
\end{abstract}
\section{Introduction}
In this note we discuss the geometry of solutions to a \p-harmonic
system of equations in the plane. That is, for a map
$u=(u^1,u^2):\Omega\subset \mathbb{R}^2\rightarrow \mathbb{R}^2$ and
$1<p<\infty$ we will investigate the following nonlinear system of
equations:
\begin{equation*}\label{system_1}
\div(|Du|^{p-2}Du)=0,
\end{equation*}
where $Du$ stands for the Jacobi matrix of $u$ and $|Du|^2=|\nabla u^1|^2+|\nabla u^2|^2$. If a solution exists
it is called a \p-harmonic map. The system originates from the
Euler-Lagrange system for the energy $\int_{\Om} |Du|^p$ and
therefore, the natural domain of definition for solutions is the
Sobolev space $W^{1, p}_{loc}(\Om, \R^n)$. However, in the
discussion below we will deal mainly with $C^2$-regular maps.
Equivalently, this system can be written as follows.
\begin{equation}\label{system_2}
 \left\{
\begin{array}{l}
{\rm div}(|Du|^{p-2}\nabla u^1)= 0 \\
\\
{\rm div}(|Du|^{p-2}\nabla u^2)= 0.
\end{array} \right.
\end{equation}
Furthermore, for $p=2$ the system reduces to the harmonic one and so from that point
of view \p-harmonic maps are the nonlinear counterparts of harmonic
transformations. On the other hand, if map $u$
degenerates to a single function $u=(u^1, 0)$, we retrieve from (\ref{system_2})
the classical \p-harmonic equation $\div (|\nabla u^1|^{p-2}\nabla u^1)=0$. In spite of similarity to the definition of the
\p-harmonic equation, the \p-harmonic system is far more
complicated, as the component functions are tangled together by the
appearance of $Du$ in both equations. This property, together with
degeneracy of the system at points where $Du=0$ makes the analysis
of \p-harmonics difficult and challenging.\par
  The \p-harmonic operators and systems arise naturally in a variety of applications
e.g. in nonlinear elasticity theory \cite{iko11, io11}, nonlinear
fluid dynamics \cite{am05, dpr06}, as well as in cosmology or
climate sciences and several other areas (see e.g. \cite{a1} and
references therein). In pure mathematics the \p-harmonic maps appear
for instance in differential geometry \cite{hl95, wa05, we08} or in
relation to differential forms and quasiregular maps
\cite{bonkhein}.
 In what follows we will confine our discussion to maps between planar domains. The reason for this is twofold. First, already in the two dimensional setting computations for nonlinear Laplace operators are complicated and in higher dimensions the complexity increases even further. The second reason is related to the fact that we will often appeal to relations between the complex gradients of coordinate functions
of a \p-harmonic map and quasiregular maps. Such relations known for
planar \p-harmonic equation \cite{bi, 3} has been recently
established also in the setting of \p-harmonic maps (see \cite{a1}
or discussion in Section 2 and Appendix \ref{App-system}). The
corresponding relations between nonlinear PDEs and quasiregular maps
beyond the plane remains an open problem.

We recall that in the planar case  quasiregular map can be defined in terms of the Beltrami coefficient $\mu$. Namely, a map $F$ is quasiregular if there exists a constant $k$ such that
\begin{equation}\label{beltrami}
 |\mu|=\frac{|F_{\overline{z}}|}{|F_z|}\leq k<1 \quad \hbox{ a.e. in } \Om.
\end{equation}
For the equivalent definitions of quasiregular maps and further information on this topic we refer to e.g. \cite{lv73}, \cite[Chapter
14]{hkm}, \cite[Chapter 3]{6}. Other properties of quasiregular maps needed in our presentation
will be recalled throughout the discussion.\par

The subject of our interest will be the geometry of \p-harmonic
surfaces, that is the geometry of the graphs of coordinate functions
of planar \p-harmonic maps. The main difficulty lies in the fact
that functions $u^1$ and $u^2$ are coupled by $Du$, and so many of our
estimates involve both coordinates and depend on the Jacobi matrix norm $|Du|$.\par
In Section 2 we recall and introduce various representations of the \p-harmonic operator and
\p-harmonic system needed in further sections, as depending on the
discussed problem we will adopt different points of view on
\p-harmonicity.\par In Section 3 we show that for some range of
parameter $p$ the positivity of Hessian determinant for one
coordinate function of $u$ implies that the second Hessian
determinant is negative. Such a phenomenon has not been noticed
before for \p-harmonic maps. From this observation we infer number
of conclusions regarding convexity of coordinate functions and their
level sets and the Gauss curvature of the corresponding surfaces. In
the latter case, we generalize work of Lindqvist \cite{lin1} on
\p-harmonic surfaces. Using the class of radial maps we illustrate
Section 3 by example postponed to Appendix \ref{radial-hessian} due
to complexity and technical nature. \par Section 4 is devoted to
studying the curvature of level curves. Following the ideas of
Alessandrini \cite{al} and Lindqvist \cite{lin1} we prove Theorem
\ref{k-lemma} providing the local estimates of lengths of a level
curves of $u^1$ and $u^2$. To our best knowledge such estimates in
the nonlinear vectorial setting are not present in the literature so far.
\par We continue investigation of level curves in Section 5, where
basing on the work of Talenti \cite{tal} we discuss level curves of
steepest descent and provide some estimates for the curvature
functions involving both the level curves and their orthogonal
trajectories. Results in Sections 3-5 are based on techniques
developed in earlier work by the author \cite{a1} and therefore, for
the sake of completeness and for the readers convenience we recall
the necessary results from \cite{a1} in Appendix \ref{App-system}.
There we also extend some of the estimates from \cite{a1} due to
$C^2$ assumption on \p-harmonic maps.\par Section 6 contains
discussion of an isoperimetric inequality for \p-harmonic maps and
generalizes works of Laurence \cite{la} and Alessandrini \cite{al2}
to the setting of vector transformations. Again, in the setting of
systems of coupled differential equations such a result is new.\par

We believe that our approach based on mixture of complex analysis,
theory of quasiregular maps and PDEs techniques can be extended to
some other nonlinear systems of equations in the plane.

\section{Representations of \p-harmonic equations and
systems}\label{section2}
 In this section we recall and develop various representation formulas for
 \p-harmonic operator and system in the plane.
 The presentation is of mainly technical nature and the results here will serve as auxiliary
 tools for the discussion in the following sections. Also, our goal is to compare
 \p-harmonic transformations with their scalar counterparts (\p-harmonic
 functions) and, therefore, illustrate the differences between these cases. We would like to emphasize that despite formal similarity to the scalar case one should expect new phenomena in the setting of \p-harmonic mappings.\par
 Let $u=(u^1,u^2):\Omega \subset \mathbb{R}^2\rightarrow \mathbb{R}^2$ be a planar \p-harmonic mapping. Assume that
  $u\in~\!C^2(\Omega)$.
 Following notation in \cite{a1} we denote by $f,g$ the complex
 gradients of the first and the second coordinate function of $u$,
 respectively.
\begin{equation}\label{complex}
f\,=\,\frac{1}{2}\,(u^1_x-\,i\,u^1_y),\quad
g\,=\,\frac{1}{2}\,(u^2_x-\,i\,u^2_y).
\end{equation}
In what follows we will frequently appeal to the following equations
for $|f|, |g|, f_z, g_z, \fb, \gb$.
\begin{align}\label{cpx-grad}
 |f|&=\tfrac12|\nabla u^1|,\qquad  |g|=\tfrac12|\nabla u^2|,\nonumber \\
 f_z&=\tfrac{1}{4}\left(u_{xx}^1-u_{yy}^1-2iu_{xy}^1\right),\qquad  \fb=\overline{\fb}=\tfrac{1}{4}\left(u_{xx}^1+u_{yy}^1\right) =\tfrac14 \Delta u^1, \\
 g_z&=\tfrac{1}{4}\left(u_{xx}^2-u_{yy}^2-2iu_{xy}^2\right),\qquad \gb=\overline{\gb}=\tfrac{1}{4}\left(u_{xx}^2+u_{yy}^2\right)  =\tfrac14 \Delta u^2.\nonumber
\end{align}
Next, we define the \p-harmonic operator and express it by using the
complex notation. Let $\Omega\subset \mathbb{R}^2$ and $v\in
C^2(\Omega,\mathbb{R})$ be a Sobolev function for a given $1<p<\infty$.
The following differential operator is called the scalar \p-harmonic operator:
\begin{align}
\Delta_p v&=\div(|\nabla v|^{p-2}\nabla v) \nonumber \\
 &=|\nabla v|^{p-4}\big(|\nabla v|^2\Delta v\,+\,(p-2)(\,(v_x)^2v_{xx}+2v_xv_yv_{xy}+(v_y)^2v_{yy}\,)\big)  \label{phf} \\
&=|\nabla v|^{p-4}\big(|\nabla v|^2\Delta
v\,+\,\tfrac{p-2}{2}\langle \nabla v, \nabla |\nabla
v|^2\rangle\big). \nonumber
\end{align}
We will also need the complex formulation of a scalar \p-harmonic
operator in the plane. Using (\ref{complex}) with (\ref{phf}) (with the abuse of notation that here $f=\frac{1}{2}\,(v_x-\,i\,v_y)$) we obtain that
\begin{align}
 \div(|\nabla v|^{p-2}\nabla v)&=\tfrac{\partial}{\partial x}\left(|\nabla v|^{p-2}v_x\right)+\tfrac{\partial}{\partial y}\left(|\nabla v|^{p-2}v_y\right) \nonumber \\
&=\tfrac{\partial}{\partial x}\left(2^{p-2}|f|^{p-2}(f+\barf)\right)+i\tfrac{\partial}{\partial y}\left(2^{p-2}|f|^{p-2}(f-\barf)\right) \nonumber \\
&=2^p \Re\left(|f|^{p-2}f\right)_{\overline{z}} =2^{p-1}\left(\,\left(|f|^{p-2}f\right)_{\overline{z}}+\overline{\left(|f|^{p-2}f\right)_{\overline{z}}}\,\right)\nonumber \\
&=2^{p-2}|f|^{p-2}\left(2p\fb +
(p-2)\left(\tfrac{f}{\barf}\overline{f_z}+\tfrac{\barf}{f}f_z\right)\right).\label{cpx-phf}
\end{align}
Let us now focus our attention on \p-harmonic maps. Using the
definition of the \p-harmonic operator (\ref{phf}) the following
form of the \p-harmonic system (\ref{system_2}) in the plane can be
established at the points where $\nabla u^1\not =0$ and $\nabla
u^2\not =0$:
\begin{equation}\label{grad_rep}
 \left\{
\begin{array}{l}
|\nabla u^2|^{p-4}\Delta_pu^1\,+\,|\nabla u^1|^{p-4}{\rm div}(|\nabla u^2|^{p-2}\nabla u^1)\,=\,0 \\
\\
|\nabla u^1|^{p-4}\Delta_pu^2\,+\,|\nabla u^2|^{p-4}{\rm
div}(|\nabla u^1|^{p-2}\nabla u^2)\,=\,0.
\end{array} \right.
\end{equation}
\begin{rem}
 The analogous representation can be stated in any dimension $n\geq 2$.
 However, here we confine our discussion to the case $n=2$ only.
\end{rem}
\noindent Equivalently, system (\ref{grad_rep}) can also be written
as follows.
\begin{equation}\label{grad_rep2}
 \left\{
\begin{array}{l}
|\nabla u^1|^{4-p}\Delta_pu^1\,+\,|\nabla u^2|^2\Delta u^1\,+\,\frac{p-2}{2}\langle\,\nabla u^1, \nabla |\nabla u^2|^2\,\rangle\,=\,0 \\
\\
|\nabla u^2|^{4-p}\Delta_pu^2\,+\,|\nabla u^1|^2\Delta
u^2\,+\,\frac{p-2}{2}\langle\,\nabla u^2, \nabla |\nabla
u^1|^2\,\rangle\,=\,0.
\end{array} \right.
\end{equation}
We will sketch the proof only for the first equation of system (\ref{grad_rep}) as the derivation of the second equation goes along the same lines.
\begin{align}
\div(|Du|^{p-2}\nabla u^1)=\tfrac{\partial}{\partial x}\left((|\nabla u^1|^2+|\nabla u^2|^2)^{\frac{p-2}{2}}u^1_{x}\right)+\tfrac{\partial}{\partial y}\left((|\nabla u^1|^2+|\nabla u^2|^2)^{\frac{p-2}{2}}u^1_{y}\right)&=0\nonumber \\
|\nabla u^1|^2\Delta u^1+|\nabla u^2|^2\Delta
u^1+\tfrac{p-2}{2}\left(\left(|\nabla u^1|^2\right)_x u^1_x+
\left(|\nabla u^2|^2\right)_x u^1_x +\left(|\nabla u^1|^2\right)_y u^1_y+\left(|\nabla u^2|^2\right)_y u^1_y\right)&=0\nonumber \\
|\nabla u^1|^2\Delta u^1+\tfrac{p-2}{2}\langle \nabla u^1,
\nabla|\nabla u^1|^2\rangle +
|\nabla u^2|^2\Delta u^1+\tfrac{p-2}{2}\langle \nabla u^1, \nabla|\nabla u^2|^2\rangle&=0 \nonumber \\
 |\nabla u^2|^{p-4}\Delta_pu^1 + |\nabla u^1|^{p-4}(|\nabla u^2|^{p-2}\Delta u^1+(p-2)|\nabla u^2|^{p-4})&=0 \nonumber \\
 |\nabla u^2|^{p-4}\Delta_pu^1 + |\nabla u^1|^{p-4}\left(\tfrac{\partial}{\partial x}(|\nabla u^2|^{p-2}u^1_x)+\tfrac{\partial}{\partial y}(|\nabla u^2|^{p-2}u^1_y)\right)&=0. \label{phm-repr-sec}
\end{align}
Using the definition of the divergence operator we arrive at the
first equation of (\ref{grad_rep}). Multiplying (\ref{phm-repr-sec})
by $|\nabla u^1|^{4-p}|\nabla u^2|^{4-p}$  we produce the first
equation of (\ref{grad_rep2}). Similar computations allow us to
obtain second equation of (\ref{grad_rep}) and (\ref{grad_rep2}),
respectively.\par The complex notation and the connection between
PDEs in the plane and functions of complex variable is nowadays
classical and has become very fruitful and brought lots of insight
into both complex analysis and theory of differential equations, to
mention for instance the theory of Beltrami equation or the theory
of generalized analytic functions (see e.g. \cite{6, ve}). It turns
out that also in the setting of \p-harmonic systems such relations
can be discovered. Indeed, in \cite{a1} we proved that $f$ and $g$
satisfy the following system equivalent to (\ref{system_2}). We will
frequently appeal to this result and its consequences in further
sections.\par
\begin{theorem}[Theorem 1, \cite{a1}]\label{jeden}
 For $1<p<\infty$ let $u=(u^1, u^2)$ be a $C^2(\Omega,\mathbb{R}^2)$ \p-harmonic
 mapping. Consider complex gradients $f,g$ of coordinate functions
 $u^1,u^2$, respectively (eqs. (\ref{complex})). We have the following system of quasilinear equations.
 \begin{equation}\label{system}
\left\{
\begin{array}{ll}
  \!\!\bigg(2p+\frac{4|g|^2}{|f|^2}\bigg)f_{\overline{z}} &\!\!\! =
  (2-p)\bigg(\,\frac{\overline f}{f}\,f_{z}\,+\,\frac{f}{\overline f}\,\overline{f_{z}}\,\bigg) \\
&\!\! +\, (2-p)\bigg[\frac{\overline g}{f}\,g_{z}+\frac{g}{\overline
f}\overline{g_{z}}+\bigg(\frac{\overline{g}}{\overline{f}}+\frac{g}{f}
\bigg)g_{\overline{z}} \bigg]  \\
\bigg(2p+\frac{4|f|^2}{|g|^2}\bigg)g_{\overline{z}} &\!\!\! =
(2-p)\bigg(\,\frac{\overline g}{g}\,g_{z}
+\frac{g}{\overline g}\,\overline{g_{z}}\,\bigg) \\
&\!\! + (2-p)\bigg[\frac{\overline f}{g}\,f_{z}\,
+\frac{f}{\overline
g}\overline{f_{z}}+\bigg(\frac{\overline{f}}{\overline{g}}+\frac{f}{g}
\bigg)f_{\overline{z}} \bigg]
\end{array} \right.
\end{equation}
at the points where $f\not=0$ and $g\not=0$.
\end{theorem}
\begin{rem}
 In \cite{a1} the above result is proven for $p\geq 2$. Here,
the $C^2$ assumption on $u$ allows us to extend theorem to the whole
range of $1<p<\infty$ (see Appendix \ref{App-system} for further
discussion).
\end{rem}
Let us compare the above system to its scalar counterpart. For that
purpose, note that if $u^2\equiv 0$, then the mapping $u$ reduces to
one coordinate function $u^1$. In such a case system (\ref{system})
reduces to well known equation, see e.g. equation (5) in \cite{bi}:
$$f_{\overline{z}}=
  \bigg(\frac{1}{p}-\frac{1}{2}\bigg)\bigg(\,\frac{\overline f}{f}\,f_{z}\,+\,\frac{f}{\overline f}\,\overline{f_{z}}\,\bigg).$$

\noindent From this we immediately infer that $f$ is a quasiregular mapping (more on this topic can be found in \cite {6}).
\par

Let us also mention that system (\ref{system}) can be solved for
$f_{\overline{z}}$ and $g_{\overline{z}}$. As a result we arrive at
the following representation for $f$ and $g$. (We refer to
Appendix~\ref{App-system} and discussion in \cite{a1} for the
definition of matrix $A(f,g)$ and further estimates).
\begin{equation}\label{operator}
\left[ \begin{array}{c}
 f \\
 g
\end{array} \right]_{\overline{z}} =
A(f,g) \left[ \begin{array}{c}
 f \\
 g
\end{array} \right]_{z}\,+
\overline{ A(f,g)} \overline {\left[ \begin{array}{c}
 f \\
 g
\end{array} \right]_{z}
 }.
\end{equation}

The ellipticity of such quasilinear system has been proven in
Theorem 2 in \cite{a1}. There, we also showed that, perhaps
surprisingly, the coefficients of $A(f, g)$ can be estimated in terms of
parameter $p$ only.

 Using the above system one can investigate when $f$ and $g$ are
 quasiregular maps, extending the results known for \p-harmonic
 functions in the plane (Section 3 in \cite{a1}).\par
 We would like to add that in the planar case the relation
 between quasiregular mappings and \p-harmonic functions is known
 in much deeper details than we just sketched it above. For
 instance one can prove that the coordinates of a planar quasiregular
 map, as well as the logarithm of the modulus of such map satisfy
 certain elliptic equations and the same holds for the logarithm of the
 modulus of the gradient of \p-harmonic function (see \cite{3}). The higher-dimensional counterparts
 of such properties remain unknown neither for \p-harmonic functions nor \p-harmonic mappings,
 due to lack of Sto\"ilow factorization beyond the complex plane.\par

\section{Convexity of coordinate functions, the Gaussian curvature of \p-harmonic surfaces}
 Below we use system (\ref{operator}) and the estimates
for the entries of matrix $A(f,g)$ (see (\ref{A-entries}) in
Appendix~\ref{App-system}) to determine mutual relations between
convexity of coordinate functions of \p-harmonic map. We discover an
interesting phenomenon that for a certain range of $p$ convexity of
one coordinate function implies the concavity of the other.
Theorem~\ref{hessian} has not been noticed before in the literature,
mainly due to the lack of enough wide classes of examples of
\p-harmonic maps. Among equivalent formulations of
Theorem~\ref{hessian} and its corollaries we discuss the sign of
Gauss curvature for \p-harmonic surfaces and convexity of their
level sets. \par

Since the convexity properties of a function are govern by the
second derivatives matrix of such function, the Hessian matrices for
$u^1$ and $u^2$ and the analysis of their signs will be of our main
interest. Using equations (\ref{cpx-grad}) we express the determinant of Hessian
$H(u^1)$ as follows:
\begin{equation*}
 \det H(u^1)=u^1_{xx}u^1_{yy}-(u^1_{xy})^2=4(|\fb|^2-|f_z|^2),
\end{equation*}
where $f$ is a complex gradient of $u^1$ (see (\ref{complex})). Related is the Gauss curvature of a surface $z=u^1(x,y)$.
\begin{equation}\label{Gauss-curv-hess}
K_{u^1}=\frac{u^1_{xx}u^1_{yy}-(u_{xy}^1)^2}{(1+(u^1_x)^2+(u^1_y)^2)^2}=\frac{\det H(u^1)}{(1+(u^1_x)^2+(u^1_y)^2)^2}.
\end{equation}
The similar formulas hold for $H(u^2)$ and $K_{u^2}$.
\begin{theorem}\label{hessian}
 Suppose $u=(u^1,u^2)$ is a \p-harmonic mapping and let $p\in [\frac43, 2+\sqrt{2}]$.  If $\det H(u^2)\geq 0$, then $\det H(u^1)\leq 0$.

 Moreover, we have that if $\det H(u^2)\geq 0$ $(\det H(u^1)\geq 0, respectively)$ holds in the whole domain $\Om$, then the Gauss curvature $K_{u^1}\leq 0$ $(K_{u^2}\leq 0,  respectively)$ in $\Om$.
\end{theorem}
Before presenting the proof we will compare this observation to the
case of single \p-harmonic equation and discuss some consequences of
the above result.

\begin{rem}
 If $u^2\equiv 0$ ($u^1\equiv 0$) then \p-harmonic system (\ref{system_2}) reduces to a single \p-harmonic equation for
 $u:= u^1$ ($u:= u^2$, respectively). In such a case from Theorem~\ref{hessian} we retrieve first part of the assertion of Theorem
5.3 in \cite{lin1} which stays that for \p-harmonic surfaces
$K_u\leq 0$. Furthermore, the quasiregularity of the complex
gradient of $u$ implies that $K_u=0$ at most at isolated points or
$u$ is an affine function (cf. \cite{lin1}).
\end{rem}
 Recall that the Jacobian of $f$ satisfies $J(z, f)= -\det H(u^1)$ (similarly, $J(z, g)= -\det H(u^2)$). This,
 together with the characterization of quasiregularity via the Beltrami coefficient (\ref{beltrami}) leads us to the equivalent formulation of Theorem \ref{hessian}.
\begin{cor}
Under the assumptions of Theorem \ref{hessian} it holds that if $ J(z, g)\leq 0$, then $J(z, f)\geq 0$.
\end{cor}
\begin{rem}
 Taking into account that if $J(z,g)\leq 0$ in $\Om$, then $\barg$ is quasiregular, Theorem \ref{hessian} can be equivalently rephrased as follows:
\emph{if $\barg$ is quasiregular, then so is $f$.}
\end{rem}
 Theorem \ref{hessian} allows us also to explore the convexity properties of \p-harmonic surfaces.
 Recall, that if a function $v$ is convex, then level sets
 $\{v \leq c\}$ are convex as well.
 Similarly, if $v$ is concave, then level sets $\{v\geq c\}$ are concave.
 The analysis of convexity of level sets has been the subject of several
 interesting papers, for instance due to Kawohl \cite[Chapter 3]{ka}, Lewis \cite{lew3} or recently Ma et al. \cite{moz}, to mention only some. 
\begin{cor}
 Suppose that the assumptions of Theorem \ref{hessian} hold. If $u^2$  is a convex component function of $u$, then
 $u^1$ is concave and so are level sets $\{u^1\geq c\}$, provided that $u^1_{xx}\leq 0$.
\end{cor}
 The assertion of Theorem \ref{hessian} can also be related to work of Laurence \cite{la} on the derivatives of geometric functionals emerging in analysis, such as the length of the level curves (Laurence's work will be discussed in greater details in Section 6 below).
 The results of \cite{la} specialized to our setting read as follows.\par
\begin{cor}\label{laurence}
Let $u$ satisfies the assumptions of Theorems 1, 2  and 5 in \cite{la} and Theorem \ref{hessian} above.
Denote by $L(s)=\int_{\{u^2=s\}}{\rm d}{\mathcal H}^1$ the length of level
curve of $u^2$ corresponding to $s$.
It holds, that if $u^1$  is convex, then $L''(s)>0$.
\end{cor}
\begin{proof}[Proof of Theorem \ref{hessian}]
From formula (\ref{operator}) we get the following
important estimate (see also Remark~\ref{rem_A} and
inequality (15) in \cite{a1}):
\begin{align}
|f_{\overline{z}}| &\leq
|A_{11}(f,g)||f_z|+|A_{12}(f,g)||g_z|+|\overline{A_{11}(f,g)}||\overline{f_z}|+
|\overline{A_{12}(f,g)}||\overline{g_z}| \nonumber \\
& \leq 2A_p\left(|f_z|+|g_z|\right).\nonumber
\end{align}
Here $A_p$ is the upper bound for the entries of matrix $A(f,g)$
(see Appendix~\ref{App-system} for computations):
\begin{equation}\label{local_Ap}
A_p =
\begin{cases}
 \frac{2-p}{2p} & \hbox{for}\quad 1<p<2,  \\
 \frac{p-2}{2p} & \hbox{for}\quad 2\leq p\leq3,  \\
 \frac{(p-2)(p-1)}{4p} & \hbox{for}\quad 3<p.
\end{cases}
\end{equation}
This, together with the analogous inequality for $|\gb|$ and arithmetic-geometric mean inequality results in the estimate:
\begin{equation}\label{conv-est}
 |\fb|^2+|\gb|^2\leq 4 (2A_p)^2(|f_z|^2+|g_z|^2).
\end{equation}
With the above notation we infer from (\ref{conv-est}) the following
chain of estimates.
\begin{align}
 &\frac{1}{4}\left(\det H(u^1)-\det H(u^2)\right) = (|\fb|^2-|f_z|^2)-(|\gb|^2-|g_z|^2) \nonumber \\
 &\leq 16A_p^2 (|f_z|^2+|g_z|^2)-2|\gb|^2+|g_z|^2-|f_z|^2\nonumber \\
 &= \left(16A_p^2-1\right)|f_z|^2+\left(16A_p^2+1\right)|g_z|^2-2|\gb|^2 \nonumber \\
 &\hbox{(next, we use explicit formulas (\ref{cpx-grad}) for $g_z$ and $\gb$)} \nonumber \\
 &= \left(16A_p^2-1\right)|f_z|^2+ \frac{1}{16}\left(16A_p^2+1\right)((u_{xx}^2)^2+(u_{yy}^2)^2-2u_{xx}^2u_{yy}^2+4(u_{xy}^2)^2)
- \frac{1}{8}(\Delta u^2)^2 \nonumber \\
&= \left(16A_p^2-1\right)|f_z|^2 +\frac{1}{16}\left(16A_p^2-1\right)(\Delta u^2)^2-\frac{1}{4}\left(16A_p^2+1\right) u_{xx}^2u_{yy}^2 +\frac14\left(16A_p^2+1\right)(u_{xy}^2)^2  \nonumber \\
& = \left(16A_p^2-1\right)|f_z|^2
+\frac{1}{16}\left(16A_p^2-1\right)(\Delta
u^2)^2-\frac{1}{4}\left(16A_p^2+1\right)\det H(u^2). \nonumber
\end{align}
It follows that
\begin{equation}\label{hessian-eq2}
 \det H(u^1)\leq 4(16A_p^2-1)\left(|f_z|^2 +\frac{1}{16}(\Delta u^2)^2\right)-16A_p^2\det H(u^2).
\end{equation}
Computations involving the appropriate values of $A_p$ (see
(\ref{local_Ap}) and (\ref{A-entries}) in
Appendix~\ref{App-system}) give us that
\begin{equation}
16A_p^2\leq 1 \quad\hbox{if}\quad
\begin{cases}
 (4-3p)(4-p)\leq 0 & \hbox{for}\quad 1<p<2, \nonumber \\
 (p-4)(3p-4)\leq 0 & \hbox{for}\quad 2\leq p\leq3, \nonumber \\
 (p^2-4p+2)(p^2-2p+2)\leq 0 & \hbox{for}\quad 3<p.
\end{cases}
\end{equation}
From these conditions we derive that $16A_p^2\leq 1$ holds provided
$p\in [\frac43, 2+\sqrt{2}]$. From this and (\ref{hessian-eq2})
the first assertion of theorem follows immediately.

The second assertion of the theorem is the straightforward consequence of the first part and equation (\ref{Gauss-curv-hess}).
\end{proof}

\begin{rem}\label{hessian-rem-rad}
 The range of parameter \p~in the assertion of Theorem~\ref{hessian} is
the consequence of estimates for entries of the matrix $A(f,g)$. In
Appendix \ref{radial-hessian} we discuss the counterexample to
Theorem \ref{hessian} for some \p~outside the interval $[\frac43,
2+\sqrt{2}]$. The problem of finding such examples is the general
feature of \p-harmonic world, as we know only few classes of
\p-harmonic maps and few explicit solutions of the \p-harmonic
system of equations, namely affine, radial and quasiradial (see
\cite[Chapter 2]{phd} for the definition of the latter one class of
mappings).\par
\end{rem}
\begin{opprb}\rm
 Let $u=(u^1,\ldots, u^n)$ be a non-trivial \p-harmonic map between
domains in $\mathbb{R}^n$ for $n\geq 3$ (that is $u$ is not an
affine or constant map). Suppose that $u^i$ is convex for some
$i=1,\ldots, n$ (and so $\det H(u^i)> 0$). Does it then hold that
$\det H(u^j)\leq 0$ for $j \not=i$? Describe the conditions for
concavity of $u^j$ for $j \not=i$.
\end{opprb}
%
%

\section{The curvature of level curves}
 Below we discuss various curvature functions of level curves for
the component functions of a map $u=(u^1, u^2)$ and employ such
curvatures to estimate the length of the level curves. It appears
that such estimates require integrability of Hessians or
quasiregularity of complex gradients of $u^1$ and $u^2$. Therefore,
the complex linearization of \p-harmonic system (\ref{operator})
comes in handy. The results below extend the work of Lindqvist
\cite{lin1} for a \p-harmonic equation.\par
 Let $\{u^1=c\}$ be a nonempty level curve with the property that none of the critical
points of $u^1$ lies on this level curve. The curvature function
$k_{u^1}$ of  $\{u^1=c\}$ can be computed by the following formula:
\begin{equation}\label{curv_k}
 k_{u^1}=-\frac{(u^1_y)^2u^1_{xx}-2u^1_x u^1_y u^1_{xy}+(u^1_x)^2u^1_{yy}}{|\nabla u^1|^3}=-\div\left(\frac{\nabla u^1}{|\nabla u^1|}\right)=-\Delta_1(u^1).
\end{equation}
 Consider the above formula for a harmonic function $v$.
Theorem 3 in Talenti \cite{tal} shows that if $v$ has no critical
points, then $\frac{k_v}{|\nabla v|}$ is harmonic and $-\ln|k_v|$ is
subharmonic. As far as we know the similar results for a single
\p-harmonic equation with $p\not=2$ are not known (see also presentation in
Section 5 below).
In the next observation we further illustrate differences between scalar and vector cases
by computing curvatures $k$ for $p$-harmonic functions and coordinate functions of $u$. Moreover, the second part of the
observation can be considered as a starting point for obtaining the
counterparts of aforementioned Talenti's results in the nonlinear
setting (see also Remark~1.5 and the discussion in \cite{moz} for
some recent developments in this topic).
%
\begin{observ}\label{obs2}
 Let $p\not =2$ and suppose that the component function $u^1$ of a \p-harmonic map $u$ has no critical points on the level curve $\{u^1=c\}$. Then
\begin{align}
k_{u^1}&= -\frac{\Delta u^1}{|\nabla u^1|} + \frac{1}{|\nabla
u^1|}\left\langle\nabla|\nabla u^1|, \frac{\nabla u^1}{|\nabla u^1|}
\right\rangle
\label{k-phm} \\
&=-\frac{p-1}{p-2}\frac{\Delta u^1}{|\nabla
u^1|}+\frac{\Delta_pu^1}{(p-2)|\nabla u^1|^{p-1}}. \label{k-phm-2}
\end{align}
Equivalently, in the complex notation $k_{u^1}$ becomes
\begin{align}
 2|f|k_{u^1}&= -2\fb+\frac{f}{\barf}\overline{f_z}+\frac{\barf}{f}f_z \label{k-cx} \\
 &=-2\left(\,\ln|f|^2\right)_{z}+\frac{f}{\barf}\overline{f_z}+3\frac{\barf}{f}f_z,  \label{k-cx-log}
\end{align}
where $f$ is a complex gradient of $u^1$ as defined in (\ref{complex}).
Similar formulas hold for the second component function $u^2$ as well.

Furthermore, if $u^2\equiv 0$, then (\ref{k-phm-2}) reduces to the following:
\begin{equation}\label{k-phf}
 k_{u_1}=-\frac{p-1}{p-2} \frac{\Delta u_1}{|\nabla u_1|}.
\end{equation}
\end{observ}
Before proving the observation, we would like to make some remarks
in order to motivate above computations and present our discussion
in the wider perspective.
\begin{rem}\label{rem-ln}
 The formula (\ref{k-cx-log}) is convenient, for instance, if one knows additionally that $f$ is a quasiregular map. In such a case function
$-\ln|f|$ solves the $A$-harmonic type equation at points where $f\not=0$, see e.g. \cite{glm}. The integral estimates which follow
from this fact will be of use for us when discussing the integrability of $k_{u_1}$.
\end{rem}
\begin{rem}
For $p=2$ the nonlinear \p-harmonic system reduces to the uncoupled system of two harmonic equations, for which the curvature functions are already present in the literature, see e.g. \cite{lin1, tal}.
\end{rem}
\begin{proof}[Proof of Observation \ref{obs2}]
Equation (\ref{k-phm}) follows immediately from the divergence
formulation of curvature (\ref{curv_k}). The same formulation used
again leads us to the following identity:
\begin{align}
 \Delta u^1 &=\div\left(\frac{\nabla u^1}{|\nabla u^1|}\,|\nabla u^1|\right) =-|\nabla u^1| k_{u^1}+\frac{u_x^1}{|\nabla u^1|^2}( u^1_xu^1_{xx}+u^1_{y}u^1_{xy}) +\frac{u_y^1}{|\nabla u^1|^2}( u^1_xu^1_{xy}+u^1_{y}u^1_{yy}) \nonumber \\
 & =-|\nabla u^1| k_{u^1}+\frac{(u^1_x)^2u^1_{xx}+2u^1_{x}u^1_{y}u^1_{xy}+(u^1_{y})^2u^1_{yy}}{|\nabla u^1|^2}. \label{k-phm-k}
\end{align}
Applying the definition of \p-harmonic operator (\ref{phf}) to the
last term, we express (\ref{k-phm-k}) in the following form:
\begin{equation*}
|\nabla u^1|^2 \Delta u^1=-|\nabla
u^1|^3k_{u^1}+\tfrac{1}{p-2}|\nabla
u^1|^{4-p}\Delta_pu^1-\tfrac{1}{p-2}|\nabla u^1|^2\Delta u^1.
\end{equation*}
From this, formula (\ref{k-phm-2}) follows immediately. In order to
show the complex representation of $k_{u^1}$ we use (\ref{cpx-grad})
and (\ref{cpx-phf}) together with (\ref{k-phm}) to obtain equation
from which (\ref{k-cx}) follows straightforwardly:
\begin{equation*}
 k_{u^1} = -\frac{2(p-1)}{p-2}\frac{\fb}{|f|}+\frac{2^{p-2}|f|^{p-2}\left(2p\fb + (p-2)\left(\frac{f}{\barf}\overline{f_z}+\frac{\barf}{f}f_z\right)\right)}{2^{p-1}(p-2)|f|^{p-1}}.
\end{equation*}
We show (\ref{k-cx-log}) by first observing that
$ \left(\ln (f\barf)\right)_z=\frac{f_z}{f}+\frac{\fb}{\barf}$
since $(\barf)_z=\fb$. Then, by (\ref{k-cx}) we obtain that
\begin{equation*}
 2|f|k_{u^1}=  -2\left(\frac{\fb}{\barf}+\frac{f_z}{f}\right)\barf+\frac{f}{\barf}\overline{f_z}+3\frac{\barf}{f}f_z=
-2\left(\,\ln|f|^2\right)_{z}+\frac{f}{\barf}\overline{f_z}+3\frac{\barf}{f}f_z.
\end{equation*}
 Finally, (\ref{k-phf}) follows from the observation that if $u^1$ is a \p-harmonic function, then $\Delta_pu^1=0$ and so (\ref{k-phf})
is a special case of (\ref{k-phm-2}).\par
\end{proof}

We would like now to show one of the main results of the paper,
namely the length estimates for the level curves. We follow the
approach of Talenti \cite{tal} for planar linear elliptic equations
and of Lindqvist \cite{lin1} for planar \p-harmonic functions.
\begin{theorem}\label{k-lemma}
 Let $u=(u^1, u^2)$ be a $C^2$ \p-harmonic map in the planar domain $\Om$ for $p\not =2$. Let also $C>0$ be a constant. Denote by $B=B(z_0, R)$ a ball in $\Om$ and consider a
nonempty level curve $\{u^1=c\}\cap B\not=\emptyset$. Suppose that, either
\begin{enumerate}
\item[(1)] Euclidean norms of Hessians $\|H(u^1)\|, \|H(u^2)\|$ are in $L^2(B)$ and $|f|>C$ in $B$ or
\item[(2)] $f_z, g_z\in L^2(B)$ and $|f|>C$ in $B$ or
\item[(3)] $f$ and $g$ are quasiregular in $B$ and $|f|>C$ in $B$ or
\item[(4)] $f$ and $g$ are quasiregular in $B$ and $|f(z)|>C|z-z_0|^{\alpha}$ in $B$ with $\alpha<1$.
\end{enumerate}
Then $k_{u^1}\in L^1(B)$ and the same result holds for $u^2$ with $f$ replaced by $g$ in the above assumptions.

Moreover, suppose that the singular set of $u^1$ consists of isolated critical points only and that there are no such
points in $\{u^1=c \}\cap B$. Then
\begin{equation}\label{len-est}
 L(s) \leq \int_{\Omega\cap B}|k_{u^1}|+2\pi R.
\end{equation}
The analogous estimate holds for $u^2$.
\end{theorem}
\begin{rem}
 Note that the second parts of Assumptions (1), (2) and (3)
above can be weaken, as we need $|f|>C$ to hold only on the level
curve. Then, by the continuity of $u$ there exists an open
neighborhood, where the lower bound for $|f|$ holds as well.
Therefore, in Theorem~\ref{k-lemma} it is enough to assume that
$|f|>C$ on some open neighborhood of $\{u^1=c\}$, only.
\end{rem}
\begin{proof}[Proof of Theorem~\ref{k-lemma}]
  Formula (\ref{k-cx})  implies that
\begin{equation*}
 |f| |k_{u^1}|\leq |\fb|+|f_z|.
\end{equation*}
From the linearization of \p--harmonic system in (\ref{operator}) we infer that
$f_{\overline{z}}=A_{11}(f,g)f_z+A_{12}(f,g)g_z+\overline{A_{11}(f,g)}\overline{f_z}+\overline{A_{12}(f,g)}\overline{g_z}.$
Therefore,
\begin{equation*}
 |f| |k_{u^1}|\leq (2|A_{11}(f,g)|\,+\,1) |f_z|+2|A_{12}(f,g)||g_z|.
\end{equation*}
Recall from (\ref{A-entries}) in Appendix~\ref{App-system} that
entries of matrix $A(f, g)$ can be estimated in terms of $p$ only
and hence the following inequalities hold:
\begin{align}
 |f| |k_{u^1}|&\leq \tfrac{2}{p} |f_z|+\tfrac{2-p}{p}|g_z|\,\leq \tfrac{2}{p} (|f_z|+|g_z|) \quad \hbox{ if }\,\, 1<p<2, \nonumber \\
 |f| |k_{u^1}|&\leq \tfrac{2(p-1)}{p} |f_z|+\tfrac{p-2}{p}|g_z|\,\leq \tfrac{2(p-1)}{p} (|f_z|+|g_z|) \quad \hbox{ if }\,\, 2\leq p\leq 3, \label{lemma-est}\\
 |f| |k_{u^1}|&\leq \tfrac{p^2-p+2}{2p}|f_z|+\tfrac{(p-2)(p-1)}{p}|g_z|\,\leq \tfrac{p^2-p+2}{2p}(|f_z|+|g_z|) \quad \hbox{ if }\,\, 3<p. \nonumber
\end{align}
Denote, by $A(p)$ the maximum of constants on the right hand sides
of inequalities~(\ref{lemma-est}). Then, by the H\"older inequality
we have that
\begin{equation}\label{k-lemma-holder}
 \int_{B} |k_{u_1}| \leq 2A(p) \left(\int_{B} |f_z|^2+|g_z|^2 \right)^{\frac12} \left( \int_{B} \frac{1}{|f|^2}\right)^{\frac12}.
\end{equation}
It is then clear that Assumption (1) or Assumption (2) imply the
assertion. So is Assumption (3), as if  $f$ and $g$ are quasiregular
in $B$, then $f_z, g_z \in L^2(B)$ (see e.g. \cite{3}). If
Assumption (4) holds, then integration in polar coordinates gives us
the following estimate:
\begin{equation*}
 \int_{B(z_0, R)} \frac{1}{|f(z)|^2} dz \leq \tfrac{2\pi}{C} \int_{0}^{R}r^{1-2\alpha}dr =
 \tfrac{2\pi}{C} \tfrac{1}{2(1-\alpha)}R^{2(1-\alpha)}<\infty.
\end{equation*}
Inequality (\ref{k-lemma-holder}) then implies that $\|k_{u^1}\|_{L^1(B)}<\infty$.

 By the discussion in \cite{al} (see also Theorem 4.11 in
\cite{lin1}) we know, that if function $u^1$ defined on $\Omega$ has
isolated critical points and none of them lies on the level curve
$\{x\in \Om: u^1(x)=c \}\cap G$ for $G\subset \Om$, then the following "integration by parts"
can be performed in a set $G$ (here the definition of $k_{u^1}$ in (\ref{curv_k}) is
used as well):
\begin{equation*}
-\int_{\{x\in G: u^1(x)<c \}}\!\!k_{u^1} dz=  \int_{\{x\in G: u^1(x)<c \}}
\!\!\div\left(\frac{\nabla u^1}{|\nabla u^1|}\right) dz = \int_{\{x\in
G: u^1(x)=c \}}\!ds\,+ \int_{\{x\in \partial G: u^1(x)<c\}} \left
\langle \frac{\nabla u^1}{|\nabla u^1|}, n\!\right\rangle ds,
\end{equation*}
where $n$ denotes the outer normal vector to $\partial G$. Using the notation of Corollary~\ref{laurence} we get (cf. Formula (v) in \cite{al}) that
\begin{equation}\label{l-perim}
L(c):={\rm length}(\{x\in G: u^1(x)=c \})\leq
\int_{G}|k_{u^1}|\,+\,\hbox{perimeter of }G.
\end{equation}
 Combining this inequality for $G=\Omega\cap B(z_0, R)$ with the above integrability result of $k_{u_1}$ we obtain (\ref{len-est}).
\end{proof}
\begin{rem}\label{rem-f-g}
 In Section 3 in \cite{a1} (see also Appendix \ref{App-system} below)
we discuss an inequality relating $f_z$ and $g_z$ which implies that
$f$ is quasiregular: $|g_z|\leq C(p) |f_z|$. This condition allows
us to weaken Assumptions (3) and (4) and
require only $f$ to be quasiregular. Indeed, in the proof of Theorem~\ref{k-lemma} quasiregularity of $g$ is used only to obtain $L^2$-integrability of
$g_z$.\par
 Similar simplification occurs when using the
logarithmic representation (\ref{k-cx-log}) of $k_{u^1}$. Namely, if
$f$ is quasiregular, then components of $f$ and $-\ln|f|$ satisfy
certain elliptic equation, from which the integrability of $|\nabla
\ln |f||$ can be inferred, see discussion in \cite[Section 2]{lin1}.
The $L^1$-integrability of $k_{u^1}$ then follows immediately from
$L^1$-integrability of $f_z$ (cf. Remark~\ref{rem-ln}).
\end{rem}

\begin{rem}
 Observe that under the quasiregularity assumptions (3) or (4) of
Theorem~\ref{k-lemma} the requirement for $u^1$ to have only isolated
critical points is satisfied automatically, since quasiregular maps
are discrete and open (see e.g. discussion in \cite{lin1}).
\end{rem}
\section{Level curves and curves of steepest descent}
 In the previous section we defined $k_v$, the curvature of the level curves of function $v$.
 Similarly, one may introduce function $h_v$, the curvature of the orthogonal trajectories of the level curves
 (also called lines of steepest descent):
\[
 h_v =\frac{(v_{xx}-v_{yy})v_x v_y -v_{xy}((v_x)^2-(v_y)^2)}{|\nabla v|^3}.
\]
  By considering a function
\begin{equation}\label{phi-def}
 \phi_v =k_v+i h_v=-2\,\frac{\partial}{\partial z}\left(\frac{\barf}{|f|}\right) \quad\hbox{ for } f=v_x-iv_y
\end{equation}
we obtain yet another tool to analyze the geometry of solutions of
partial differential equations, see e.g. \cite{tal, lin1}. Theorem 3
in \cite{tal} shows, that if $v$ is a harmonic function without
critical points, then $\phi_v$ satisfies certain nonlinear PDE.
Moreover, properties of $\phi_v$ can be used to show that
$\frac{k_v}{|\nabla v|}$ and $\frac{h_v}{|\nabla v|}$ are conjugate
harmonic and that $-\ln|k_v|$ and $-\ln|h_v|$ are subharmonic.
Theorem 3 in \cite{tal} has been partially extended by Lindqvist to
the nonlinear setting of \p-harmonic functions. Namely, Theorem 4.5
in \cite{lin1} asserts that for a planar \p-harmonic function $u$ it
holds that
\begin{align}
 \phi_u |f|&= |f|^2\frac{\partial}{\partial z}\left(-\frac{1}{f}\right)
 +\frac{p-2}{p}\,|f|^2\,\Re\frac{\partial}{\partial z}\left(-\frac{1}{f}\right) \label{lin-thm45} \\
 &= \tfrac{\barf}{f}f_z+\tfrac{p-2}{2p}\left(\tfrac{\barf}{f}f_z+\tfrac{f}{\barf}\overline{f_z}\right),\quad \hbox{ when } f\not=0.
 \nonumber
\end{align}
 Since $|k_v|\leq |\phi_v|$, Lindqvist employed properties of $\phi_v$ together with the theory of
 quasiregular mappings and stream functions to prove the estimates
for the length of  level curves, similar to (\ref{len-est}) above,
in terms of integral of $\phi_v$, see \cite[(4.12)]{lin1}.
\begin{theorem}\label{obs-steepest}
 If $u=(u^1, u^2)$ is \p-harmonic, then at the points where $f, g\not=0$ it holds that
\begin{equation} \label{phi-psi2}
 |\phi_{u^1}||f| \leq C(p)(|f_z|+|g_z|),\qquad |\phi_{u^2}||g|\leq C(p)(|f_z|+|g_z|).
\end{equation}
\end{theorem}
\begin{proof}
 Computing the right-hand side of (\ref{phi-def}) we get that
\begin{equation}\label{phi-eq}
 \phi_{u^1}|f|=|f|^2\left(\frac{f_z}{f^2}-\frac{\fb}{|f|^2}\right)=-\fb+\frac{\barf}{f}f_z.
\end{equation}
From (\ref{operator}) we know that
$f_{\overline{z}}=A_{11}(f,g)f_z+A_{12}(f,g)g_z+\overline{A_{11}(f,g)}\overline{f_z}+\overline{A_{12}(f,g)}\overline{g_z}.$
 Substituting this in (\ref{phi-eq}) we obtain
\begin{align}
 \phi_{u^1}|f|&=\left(\tfrac{\barf}{f}-A_{11}(f,g)\right)f_z-A_{12}(f,g)g_z-\overline{A_{11}}(f,g)\overline{f_z}-\overline{A_{12}(f,g)}\overline{g_z} \label{phi-psi}, \\
 \phi_{u^2}|g|&=A_{21}(f, g)f_z+\left(\tfrac{\barg}{g}-A_{22}(f,g)\right)g_z-\overline{A_{21}}(f,g)\overline{f_z}-\overline{A_{22}(f, g)}\overline{g_z}. \nonumber
\end{align}
 Computing $\phi_{u^2}$ and then $\gb$ from
(\ref{operator}) results in the analogous formula for
$\phi_{u^2}|g|$. The estimates for the entries of matrix $A(f,g)$ in
(\ref{A-entries}) applied to equations (\ref{phi-psi}) immediately
give us the estimates (\ref{phi-psi2}).
\end{proof}
 The case of \p-harmonic functions can now be identified as a special case of Theorem~\ref{obs-steepest}. By using matrix (\ref{Adef}) and (\ref{Adef-coeff}) in Appendix~\ref{App-system} we may
 easily check that if $u^2\equiv 0$, and hence also $g\equiv 0$, then $A_{11}(f, 0)=\tfrac{2-p}{2p}\tfrac{\barf}{f}$, whereas $A_{12}(f,0)\equiv 0$. Thus, equation (\ref{phi-psi}) reduces to (\ref{lin-thm45}).

 In Theorem~\ref{obs-steepest} we use equation (\ref{phi-def}) to
extend the aforementioned Theorem 4.5 in \cite{lin1} to vectorial setting.
Furthermore, under Assumptions (2) or (3) or (4) of
Theorem~\ref{k-lemma} we may prove the similar integrability result
for $\phi_{u^1}$ ($\phi_{u^2}$) as obtained for $k_{u^1}$
($k_{u^2}$) and, in a consequence, obtain a counterpart of level
curves length estimate (\ref{len-est}) expressed in terms of
functions $\phi_{u^1}$ ($\phi_{u^2}$), respectively. This result
generalizes Theorem 4.11 in \cite{lin1} on the integrability of
$\phi_v$ for a \p-harmonic function $v$ in the plane.\par
 We would like to emphasize that in the setting of \p-harmonic maps the
concept of stream functions (cf. \cite{ar}), used by Lindqvist \cite{lin1} to extend
the harmonic result to the nonlinear setting, is not available and
it is only due to estimates for $A(f, g)$ in (\ref{A-entries}) for
the operator form of \p-harmonic map (\ref{operator}) that we are
able to prove the above result and the mentioned counterpart of
Theorem 4.11 in \cite{lin1}.

\section{The isoperimetric inequality for the level curves}

 Below we derive a variant of an isoperimetric inequality for \p-harmonic mappings in the plane.
 In the setting of \p-harmonic functions on annuli this type results
are due to Alessandrini \cite{al2} and Longinetti \cite{long} (cf.
Remark~\ref{rem-iso2}). Our approach is based on the work by
Laurence \cite{la} and extends \cite{al2, long}. To our best
knowledge our isoperimetric inequality is new in the setting of
coupled nonlinear systems. We hope, the techniques used here can be
applied in the framework of more general systems of PDEs. In some
parts of the proof we use the complex notation, but we do not appeal
to the complex representation of \p-harmonic system (\ref{system}).
This section is, therefore, selfcontained and independent of earlier
results of the paper. Nevertheless, similarly to previous sections,
the properties of quasiregular maps will appear to be vital in
discussion.
\par \noindent
Let $v$ be a
function from $\Om \subset \mathbb{R}^2$ to $\mathbb{R}$ and define
\begin{equation*}
\Omega_{a, b}=\{x\in \Om: a<v(x)<b\},\quad -\infty\leq a<b\leq \infty.
\end{equation*}
Recall the function of length of a level curve:
\begin{equation}\label{iso-in-length}
 L(s)=\int_{\{x\in \Om\,:\,v(x)=s\}}{\rm d}\mathcal{H}^1,
\end{equation}
where ${\rm d}\mathcal{H}^1$ stands for the {$1$-Hausdorff} measure.
However, in what follows for the sake of simplicity we will often
omit the measure in notation. Theorem 1 in \cite{la} asserts, that
for a function $v\in C^3(\overline{\Om})$ such that $v$ is constant
on the boundary of $\Om$, it holds that if $|\nabla v|\geq c$ in
$\Om_{a,b}$ for some $c>0$ and given $a, b$, then
\begin{align}
 L'(s)&=\int_{\{x\in \Om\,:\,v(x)=s\}} \div\left(\frac{\nabla v}{|\nabla v|}\right)\frac{{\rm d}\mathcal{H}^1}{|\nabla v|}, \label{L-der1} \\
 L''(s)&=\int_{\{x\in \Om\,:\,v(x)=s\}} \left[\div\left(\frac{\nabla v \Delta v}{|\nabla v|^3}\right)
+\Delta\left(\frac{1}{|\nabla v|}\right) \right]\frac{{\rm d}\mathcal{H}^1}{|\nabla v|}. \label{L-der2}
\end{align}
We would like to point out that the
lower bound assumption on $|\nabla v|$ can be weaken (see part 2 of
Remark~\ref{rem-iso} below). However, in what follows we will not
explore this observation any further. In addition to (\ref{L-der1})
and (\ref{L-der2}) we will use two other interesting formulas,
holding for $C^3$ functions (cf. (1.4) and (2.6) in \cite{al}):
\begin{align}
\Delta \ln |\nabla v| &=\div\left(\frac{\Delta v}{|\nabla v|^2}\nabla v\right), \nonumber \\
\div\left(\frac{\Delta v}{|\nabla v|^3}\nabla v\right)
+\Delta\left(\frac{1}{|\nabla v|}\right) &=\left \langle
\nabla\left(\frac{1}{|\nabla v|}\right), \frac{\Delta v}{|\nabla
v|^2}\nabla v - \frac{1}{|\nabla v|}\nabla |\nabla v| \right
\rangle. \label{L-aux}
\end{align}
To this end, we will focus our discussion on the case $v=u^1$, but
the similar results hold for $u^2$ as well.
\noindent Recall that, by $f$ and $g$ we denote the complex gradient
of $u^1$ and $u^2$, respectively.
\begin{theorem}\label{iso-thm}
Let $\Om'\subset B(z_0, R)\subset B(z_0, 4R) \Subset \Om$ and let $u=(u^1, u^2)$ be a $C^3(\overline{\Om})$-regular \p-harmonic map. Suppose that the coordinate function $u^1$ is constant on $\partial \Om'$ and that there exists a positive constant $c$ such that $|\nabla u^1|>c$ in $\Om'$. Furthermore, let us assume that $f$ and $g$ are quasiregular in $\Om'$ and consider $L(s)$ in (\ref{iso-in-length}) for $v=u_1$. Then the following formulas hold at the points, where $\nabla u^1\not =0$ and $\nabla u^2\not =0$.

\noindent If $p=2$, then
\begin{equation}\label{harmonic-case}
 (\,\ln L(s)\,)''\geq 0.
\end{equation}
Otherwise, if $p\not=2$, then
\begin{equation}\label{eq-50}
L^{\tfrac{1}{p}}(s)\left(\tfrac{p}{p-1}\,L^{\frac{p-1}{p}}(s)\right)''
\geq -\frac{C}{R^{\tfrac{4}{p}}},
\end{equation}
where $C=C(c, p, \|Du\|_{L^p(B_{2R})}, {\rm dist}(\Om',\partial \Om))$ is positive.\par
%
Moreover, the equality in (\ref{harmonic-case})
and (\ref{eq-50}) is attained if $u^2\equiv 0$ or $u^1\equiv u^2$ and for $\Om'$
a circular annulus. In such a case $u^1$ is a radial \p-harmonic
function in $\Om'$ and $C=0$ in (\ref{eq-50}).
\end{theorem}
Theorem generalizes the linear case of harmonic functions, as well
as the case of \p-harmonic functions for $p\not =2$.
\begin{rem}\label{rem-iso2}

\item[1.] If $p=2$ and $\Om$ is a planar annulus, then we retrieve the well known harmonic case discussed for instance by  Laurence \cite[Theorem 6]{la} and Alessandrini \cite[Formula (1.3a), Theorem 1.1]{al2}.
\item[2.] If $p\not=2$ and $u^2\equiv 0$ or $u^1\equiv u^2$, then $u$ degenerates to a \p-harmonic function $u^1$.
 In such a case, the analysis of steps of the proof below allows us to
retrieve the nonlinear part of assertion in \cite[Formula (1.3b),
Theorem 1.1]{al2} with $\Lambda=p-1$. In \cite{al2}, it is assumed
that $\Om$ is an annuli. This is because for such $\Om$ it can be
showed that the gradient norm is strictly positive (the argument
goes back to Lewis \cite{lew3}), and therefore
(\ref{L-der1}) and (\ref{L-der2}) can be applied. Instead, in
Theorem~\ref{iso-thm} we localize the discussion on the subset
$\Om'\Subset \Om$ and assume that $|\nabla u^1|>c>0$ (see also the
discussion of equality in the proof of Theorem~\ref{iso-thm}).
\item[3.] In Theorem \ref{iso-thm} we may assume that $s\leq {\rm max}_{_{\partial \Om'}} u^1$
 due to the maximum principle for coordinates of \p-harmonic maps
(see Observation~\ref{comp-princ} in Appendix \ref{App-max}).
\end{rem}
\noindent Let us comment the assumptions and hypothesis of the
theorem.
\begin{rem}\label{rem-iso}
\item[1.] The $C^3$-regularity of $u$ is assumed in order to be able
to apply formulas (\ref{L-der1}) and (\ref{L-der2}) for $L'$ and $L''$, respectively.
\item[2.] According to Remark on pg. 266 in \cite{la}, the assumption that
$|\nabla u^1|>c$ can be weaken due to the Sard theorem. Namely, one
can require formulas (\ref{L-der1}) and (\ref{L-der2}) to hold only
for the regular values of $u^1$. Since, let $t$ be a regular value
of $u^1$. Then there exist $\epsilon>0$ and $c>0$ such that $|\nabla
u^1|>c$ on the set $\{x\in \Om: t-\epsilon<u^1(x)<t+\epsilon\}$.
Furthermore, by Remark~2 on pg. 267 in \cite{la}, the assumption on
lower bounds for $|\nabla u^1|$ can be replaced by the integrability
condition of a suitable power of the gradient of $u^1$.
\item[3.] We require $\Om'$ to be enough far away from the boundary of
$\Om$, since in the proof we use the following important estimate
for \p-harmonic maps due to Uhlenbeck \cite[Theorem 3.2]{4}:
\begin{equation}\label{uhl-est}
 \sup_{B_R}|Du|\leq \frac{C}{R^\frac{2}{p}} \|Du\|_{L^p(B_{2R})},
\end{equation}
where $C(p)$ is a constant in the Sobolev imbedding theorem.
\item[4.]
 One of the assumptions of Theorem \ref{iso-thm} is that $f$ and $g$
are quasiregular maps. Using computations similar to Formula (17) in
\cite[Section 3]{a1} we may determine conditions under which complex
gradients $f$ and $g$ of coordinate functions of a \p-harmonic map
are quasiregular (cf. Remark~\ref{rem-f-g}).
\end{rem}
\begin{proof}[Proof of Theorem \ref{iso-thm}]
 The definition of the \p-harmonic operator (\ref{phf}) and the
representation of the \p-harmonic system (\ref{grad_rep2}) allow us
to write the first equation of such system as follows:
\begin{equation}\label{Lapl}
 0=|\nabla u^1|^2\left(\Delta u^1+(p-2)\left \langle\nabla |\nabla u^1|, \frac{\nabla u^1}{|\nabla u^1|}\right \rangle\right)+|\nabla u^2|^2\left(\Delta u^1+(p-2)\left \langle \nabla u^1, \frac{\nabla|\nabla u^2|}{|\nabla u^2|}\right\rangle \right).
\end{equation}
From this equation we compute the Laplacian of $u^1$ and use
(\ref{curv_k}), (\ref{k-phm}) and (\ref{L-der1}) to obtain the
formula for $L'(s)$:
\begin{align}
 L'(s)&= \int_{\{\,u^1=s\}}\!\!\!\div\left(\frac{\nabla u^1}{|\nabla u^1|}\right)\frac{{\rm d}s}{|\nabla u^1|}=
 \int_{\{\,u^1=s\}} \frac{1}{|\nabla u^1|}\left( \frac{\Delta u^1}{|\nabla u^1|} - \frac{1}{|\nabla u^1|}\left\langle\nabla|\nabla u^1|, \frac{\nabla u^1}{|\nabla u^1|}  \right\rangle \right) \nonumber \\
&= -\int_{\{\,u^1=s\}} \frac{1}{|\nabla u^1|^2}\Bigg[
(p-2)\frac{|\nabla u^1|^2}{|Du|^2}\left \langle\nabla |\nabla u^1|, \frac{\nabla u^1}{|\nabla u^1|}\right\rangle+
(p-2)\frac{|\nabla u^2|^2}{|Du|^2}\left \langle\nabla u^1, \frac{\nabla|\nabla u^2|}{|\nabla u^2|}\right\rangle
\nonumber \\
&\phantom{aaaaaaaaaaaaaaaa}+\left \langle\nabla |\nabla u^1|, \frac{\nabla u^1}{|\nabla u^1|}\right\rangle \Bigg]. \label{L'}
\end{align}
Next we compute $L''(s)$. Combining (\ref{L-der2}) with
(\ref{L-aux}) and (\ref{Lapl}), together with the fact that
$\nabla(|\nabla u^1|^{-1})=-\frac{\nabla |\nabla u^1|}{|\nabla
u^1|^2} $ we obtain equation:
\begin{align}
 L''(s) &=\int_{\{u^1=s\}} \left \langle \nabla\left(\frac{1}{|\nabla u^1|}\right), \frac{\Delta u^1}{|\nabla u^1|^2}\nabla u^1 - \frac{\nabla |\nabla u^1|}{|\nabla u^1|} \right \rangle
 \frac{{\rm d}s}{|\nabla u^1|}\nonumber \\
&=\int_{\{u^1=s\}} \frac{1}{|\nabla u^1|}\Bigg \langle \nabla\left(\frac{1}{|\nabla u^1|}\right), -(p-2)\frac{|\nabla u^1|^2}{|Du|^2}\left \langle\nabla |\nabla u^1|, \frac{\nabla u^1}{|\nabla u^1|}\right\rangle\frac{\nabla u^1}{|\nabla u^1|^2} \nonumber \\
&\phantom{aAAAAAAAAAAAAAAaaaa\,}-(p-2)\frac{|\nabla u^2|^2}{|Du|^2}\left \langle\nabla u^1, \frac{\nabla|\nabla u^2|}{|\nabla u^2|}\right\rangle\frac{\nabla u^1}{|\nabla u^1|^2}  - \frac{\nabla |\nabla u^1|}{|\nabla u^1|}\Bigg \rangle. \nonumber \\
&=\int_{\{u^1=s\}} \frac{1}{|\nabla u^1|^4} \Bigg[
(p-2)\frac{|\nabla u^1|^2}{|Du|^2}\left \langle\nabla |\nabla u^1|, \frac{\nabla u^1}{|\nabla u^1|}\right\rangle^2
\nonumber \\
&\phantom{AAAAAAAAAA}+(p-2)\frac{|\nabla u^2|^2}{|Du|^2}\left \langle\nabla u^1, \frac{\nabla|\nabla u^2|}{|\nabla u^2|}\right\rangle\left \langle\nabla |\nabla u^1|, \frac{\nabla u^1}{|\nabla u^1|}\right\rangle +\big|\nabla |\nabla u^1|\big|^2\Bigg]. \nonumber
\end{align}
In order to simplify the discussion, we introduce the following
notation for the terms of $L''(s)$:
\begin{align}
A_{u}&:=(p-2) \frac{|\nabla u^1|^2}{|Du|^2} \left \langle\nabla |\nabla u^1|, \frac{\nabla u^1}{|\nabla u^1|}\right\rangle^2,\nonumber \\
 B_{u}&:=(p-2)\frac{|\nabla u^2|^2}{|Du|^2}\left \langle\nabla u^1, \frac{\nabla|\nabla u^2|}{|\nabla u^2|}\right\rangle
\left \langle\nabla |\nabla u^1|, \frac{\nabla u^1}{|\nabla u^1|}\right\rangle, \label{notation} \\
 C_{u}&:=\big|\nabla |\nabla u^1|\big|^2.  \nonumber
\end{align}
With this notation $L''(s)$ reads:
\begin{equation}\label{L''-abc}
L''(s) =\int_{\{u^1=s\}} \frac{1}{|\nabla u^1|^4} \left( A_{u}+B_{u}+C_{u} \right).
\end{equation}
Using the H\"older inequality at (\ref{L'}) we obtain the following
estimate:
\begin{align}
 \frac{(L'(s))^2}{L(s)}&\leq \int_{\{u^1=s\}} \frac{1}{|\nabla u^1|^4} \Bigg[
(p-2)^2\frac{|\nabla u^1|^4}{|Du|^4}\!\left \langle\!\nabla |\nabla u^1|, \frac{\nabla u^1}{|\nabla u^1|}\right\rangle^2 +
(p-2)^2\frac{|\nabla u^2|^4}{|Du|^4}\!\left \langle\!\nabla u^1, \frac{\nabla|\nabla u^2|}{|\nabla u^2|}\right\rangle^2 \nonumber \\
& \phantom{AA}+ \left \langle\!\nabla |\nabla u^1|, \frac{\nabla u^1}{|\nabla u^1|}\right\rangle^2 + 2(p-2)^2 \frac{|\nabla u^1|^2|\nabla u^2|^2}{|Du|^4}\left \langle\nabla |\nabla u^1|, \frac{\nabla u^1}{|\nabla u^1|}\right\rangle \left \langle\nabla u^1, \frac{\nabla|\nabla u^2|}{|\nabla u^2|}\right\rangle \nonumber \\
&\phantom{AA}+2(p-2)\frac{|\nabla u^2|^2}{|Du|^2}\left \langle\nabla |\nabla u^1|, \frac{\nabla u^1}{|\nabla u^1|}\right\rangle \left \langle\nabla u^1, \frac{\nabla|\nabla u^2|}{|\nabla u^2|}\right\rangle \label{eq57} \\
&\phantom{AA}+2(p-2)\frac{|\nabla u^1|^2}{|Du|^2}\left \langle\nabla |\nabla u^1|, \frac{\nabla u^1}{|\nabla u^1|}\right\rangle^2 \Bigg]. \nonumber
\end{align}
Using notation (\ref{notation}) we express the above inequality in a
more suitable and compact form:
\begin{align}
\frac{(L'(s))^2}{L(s)}&\leq \int_{\{u^1=s\}} \frac{1}{|\nabla u^1|^4} \bigg(
(p-2)\frac{|\nabla u^1|^2}{|Du|^2} A_{u} + C_{u}+2B_{u}+2A_{u} + E_{u}
\bigg), \label{L'-est}
\end{align}
where $E_{u}$ stands for the sum of the remaining terms in formula
(\ref{eq57}):
\begin{align}
E_{u}&= 2(p-2)^2 \frac{|\nabla u^1|^2|\nabla u^2|^2}{|Du|^4}\left \langle\nabla |\nabla u^1|, \frac{\nabla u^1}{|\nabla u^1|}\right\rangle \left \langle\nabla u^1, \frac{\nabla|\nabla u^2|}{|\nabla u^2|}\right\rangle \nonumber \\
&+(p-2)^2\frac{|\nabla u^2|^4}{|Du|^4}\left \langle\nabla u^1, \frac{\nabla|\nabla u^2|}{|\nabla u^2|}\right\rangle^2 \nonumber \\ 
&= (p-2)^2\left(\frac{|\nabla u^2|^2}{|Du|^2} \left \langle\nabla u^1, \frac{\nabla|\nabla u^2|}{|\nabla u^2|}\right\rangle
+\frac{|\nabla u^1|^2}{|Du|^2}\left \langle\nabla |\nabla u^1|, \frac{\nabla u^1}{|\nabla u^1|}\right\rangle
\right)^2 \nonumber \\
&-(p-2)^2\frac{|\nabla u^1|^4}{|Du|^4}\left \langle\nabla |\nabla u^1|, \frac{\nabla u^1}{|\nabla u^1|}\right\rangle^2. \nonumber
\end{align}
By the Schwarz inequality we have that
\begin{equation*}
 \left|\frac{|\nabla u^2|^2}{|Du|^2} \left \langle\nabla u^1, \frac{\nabla|\nabla u^2|}{|\nabla u^2|}\right\rangle\right|
\leq \frac{|\nabla u^1|}{|Du|} \frac{|\nabla |\nabla u^2|^2|}{2|Du|}\leq \frac{|\nabla |\nabla u^2|^2|}{2|Du|}.
\end{equation*}
Therefore, $E_{u}$  can be estimated as follows:
\begin{align}
E_{u}&\leq (p-2)^2\left(\frac{\big|\nabla|\nabla u^2|^2\big|}{2|Du|}+\big|\nabla|\nabla u^1|\big|\right)^2
\leq 2(p-2)^2 \frac{\big|\nabla|\nabla u^2|^2\big|^2}{|\nabla u^1|^2} + 2(p-2)^2C_{u}. \nonumber
\end{align}
In the consequence, inequality $(\ref{L'-est})$ becomes:
\begin{equation}\label{L'-est-2}
\frac{(L'(s))^2}{L(s)}\leq \int_{\{u^1=s\}} \frac{1}{|\nabla u^1|^4} \bigg(
(p-2)\frac{|\nabla u^1|^2}{|Du|^2} A_{u} + C_{u}+2B_{u}+2A_{u} + 2(p-2)^2C_{u} + 2(p-2)^2 \frac{\big|\nabla|\nabla u^2|^2\big|^2}{|\nabla u^1|^2} \bigg).
\end{equation}
We may now proceed to the crucial inequality combining (\ref{L''-abc}) and (\ref{L'-est-2}).
\begin{align}\label{L'-est-3}
\frac{(L'(s))^2}{L(s)}&\leq \int_{\{u^1=s\}} \frac{1}{|\nabla u^1|^4} \bigg(
pA_{u} + 2B_{u} +  (1+2(p-2)^2)C_{u} + 2(p-2)^2 \frac{\big|\nabla|\nabla u^2|^2\big|^2}{|\nabla u^1|^2} \bigg) \nonumber \\
&=\int_{\{u^1=s\}} \frac{1}{|\nabla u^1|^4}\, p(A_{u}+B_{u}+C_{u}) \nonumber  \\
&+\int_{\{u^1=s\}} \frac{1}{|\nabla u^1|^4}
\left ((2-p)B_{u}+ (1-p+2(p-2)^2)C_{u}+2(p-2)^2 \frac{\big|\nabla|\nabla u^2|^2\big|^2}{|\nabla u^1|^2} \right).
\end{align}
In order to complete the above estimate we need the following upper
bound on $B_{u}$.
\begin{align}
 |B_{u}|&\leq |p-2|\frac{|\nabla u^1| |\nabla u^2|}{|Du|^2}\,\big|\nabla|\nabla u^1|\big|\,\big|\nabla|\nabla u^2|\big| \nonumber \\
 & \leq |p-2|\big|\nabla|\nabla u^1|\big| \frac{|\nabla|\nabla u^2|^2|}{2|\nabla u^1|} \nonumber \\
 & \leq \frac{|p-2|}{4} \big|\nabla|\nabla u^1|\big|^2 + |p-2|\frac{\big|\nabla|\nabla u^2|^2\big|^2}{4|\nabla
 u^1|^2}.\nonumber
\end{align}
Using this inequality in (\ref{L'-est-3}) we obtain:
\begin{align}\label{L'-est-4}
 \frac{(L'(s))^2}{L(s)}&\leq \int_{\{u^1=s\}} \frac{1}{|\nabla u^1|^4} \,p(A_{u}+B_{u}+C_{u}) \nonumber \\
\phantom{AAA}&+ \left(\tfrac94(p-2)^2+1-p\right)\int_{\{u^1=s\}}
\frac{C_{u}}{|\nabla u^1|^4} + \tfrac94 (p-2)^2
\int_{\{u^1=s\}}\frac{\big|\nabla|\nabla u^2|^2\big|^2}{|\nabla
u^1|^6}.
\end{align}
Upon defining
\begin{equation*}
\alpha(p):=\tfrac94 (p-2)^2+1-p\quad \hbox{and} \quad
\beta(p):=\tfrac94 (p-2)^2,
\end{equation*}
we arrive at the inequality
\begin{align*}
 \frac{(L'(s))^2}{L(s)}&\leq p L''(s) + \alpha(p)\int_{\{u^1=s\}} \frac{\big|\nabla |\nabla u^1|\big|^2}{|\nabla u^1|^4}
+\beta(p)  \int_{\{u^1=s\}}\frac{\big|\nabla|\nabla u^2|^2\big|^2}{|\nabla u^1|^6}.
\end{align*}
With such $\alpha(p)$ and $\beta(p)$, if $p=2$ (i.e. $\alpha(p)=-1, \beta(p)=0$) we retrieve (\ref{harmonic-case}), the harmonic case of the hypothesis (see also Remark~\ref{rem-iso2} above). Indeed, for $p=2$ inequality~(\ref{L'-est-4}) reads
(cf. equation (\ref{L''-abc})):
\begin{equation*}
\frac{(L'(s))^2}{L(s)}\leq \int_{\{u^1=s\}} \frac{\big|\nabla |\nabla u^1|\big|^2}{|\nabla u^1|^4}=L''(s),
\end{equation*}
and thus
\begin{equation*}
 (L'(s))^2-L(s)L''(s)\leq 0 \Leftrightarrow (\,\ln L(s)\,)''\geq 0.
\end{equation*}
Whereas, if $p\not =2$,  we have:
\begin{align}\label{iso1}
 \left(\tfrac{p}{p-1}\,L^{\tfrac{p-1}{p}}(s)\right)''\,&=\,-\tfrac{1}{p}L^{-1-\tfrac{1}{p}}(s)\left((L'(s))^2-p L(s)L''(s)\right) \nonumber \\
&\geq -L^{-\tfrac{1}{p}}(s)\left(\frac{\alpha(p)}{p}\int_{\{u^1=s\}}
\frac{\big|\nabla |\nabla u^1|\big|^2}{|\nabla
u^1|^4}+\frac{\beta(p)}{p}\int_{\{u^1=s\}}\frac{\big|\nabla|\nabla
u^2|^2\big|^2}{|\nabla u^1|^6}\right).
\end{align}
Let us now turn to estimates for the right-hand side of (\ref{iso1})
and first analyze the last term of this inequality. Using the
complex notation and the assumption that $g$, the complex gradient
of $u^2$, is quasiregular we have that
\begin{equation*}
|\nabla|\nabla u^2|^2|=8|(|g|^2)_z|\leq
8|g|(|g_z|+|\gb|)=8|g||g_z|\left(1+\tfrac{|\gb|}{|g_z|}\right)<16|g||g_z|.
\end{equation*}
Then
\begin{equation*}
 \int_{\{u^1=s\}} \frac{|\nabla|\nabla u^2|^2|^2}{|\nabla u^1|^6}\leq
 \int_{\{u^1=s\}}\frac{4|g|^2|g_z|^2}{|f|^6}.
\end{equation*}
By the assumptions, it holds that $2|f|=|\nabla u^1|>c$. From this
and from the H\"older inequality we immediately obtain the following
estimate:
\begin{equation}\label{g-term-int}
\int_{\{u^1=s\}}\frac{4|g|^2|g_z|^2}{|f|^6}\leq
\frac{256}{c^6}\,(\sup_{\{u^1=s\}}|g|^2) \int_{\{u^1=s\}}|g_z|^2.
\end{equation}
The first factor on the right-hand side can be estimated by the
Uhlenbeck inequality (\ref{uhl-est}). Moreover, the same inequality
allows us to estimate also the second integral in
(\ref{g-term-int}), as if a quasiregular transformation $g$ is
bounded (and here $|g|<2|Du|$ in $\Om$), then
$\|g\|_{W^{1,2}(\Om')}<C(\|g\|_{L^{\infty}(\Om')}, {\rm
dist}(\Om',\partial \Om))<\infty$ (see e.g. \cite{3, lv73}). Hence,
\begin{equation}\label{g-term-iso1}
 \int_{\{u^1=s\}} \frac{|\nabla|\nabla u^2|^2|^2}{|\nabla u^1|^6}\leq
\frac{256}{c^6}\,(\sup_{\{u^1=s\}}|g|^2) \int_{\{u^1=s\}}|g_z|^2
\leq \frac{C}{R^{\frac{4}{p}}}\|Du\|^2_{L^p(B_{2R})},
\end{equation}
where $C=C(c, \|Du\|_{L^p(B_{2R})}, {\rm dist}(\Om',\partial
\Om))$.\par In order to estimate the first term on the right-hand
side of (\ref{iso1}), we again use the complex notation. Then,
quasiregularity of $f$ and the assumption that $2|f|=|\nabla u^1|>c$
imply that
\begin{equation*}
\frac{\big|\nabla |\nabla u^1|\big|^2}{|\nabla
u^1|^4}=\frac{|f_z|^2}{4|f|^4}\left(1+\frac{|\fb|}{|f_z|}\right)^2
\leq \frac{1}{c^4}|f_z|^2.
\end{equation*}
Discussion similar to that in the paragraph following
(\ref{g-term-int}) leads us to inequality
\begin{equation}\label{f-term-iso1}
 \int_{\{u^1=s\}}\frac{\big|\nabla |\nabla u^1|\big|^2}{|\nabla u^1|^4} \leq  \frac{C}{R^{\frac{4}{p}}} \|Du\|^2_{L^p(B_{2R})},
\end{equation}
with $C=C(c, \|Du\|_{L^p(B_{2R})}, {\rm dist}(\Om',\partial \Om))$.
Applying inequalities (\ref{f-term-iso1}) and (\ref{g-term-iso1})
in (\ref{iso1}), we obtain the first part of assertion:
\begin{equation*}
 \left(\tfrac{p}{p-1}\,L^{\frac{p-1}{p}}(s)\right)''
\geq -\frac{C}{R^{\frac{4}{p}}}L^{-\frac{1}{p}}(s).
\end{equation*}
Here the constant $C=C(c, p, |\alpha(p)|+|\beta(p)|,
\|Du\|_{L^p(B_{2R})}, {\rm dist}(\Om',\partial \Om))$ .
Let us now discuss the case of equality in (\ref{eq-50}). Let
$u^2\equiv 0$ or $u^1\equiv u^2$. Then the assertion of theorem
reduces to the case of \p-harmonic functions previously discussed in
\cite{al2} and $L'(s)$ and $L''(s)$ take the following form (cf.
formulas in the proof of \cite[Theorem 1.1]{al2}):
\begin{align*}
L'(s)&=-\int_{\{\,u^1=s\}} \frac{p-1}{|\nabla u^1|^2}\left \langle\nabla |\nabla u^1|, \frac{\nabla u^1}{|\nabla u^1|}\right\rangle, \\
L''(s)&=\int_{\{u^1=s\}} \frac{1}{|\nabla u^1|^4} \left[
(p-2)\left \langle\nabla |\nabla u^1|, \frac{\nabla u^1}{|\nabla u^1|}\right\rangle^2
+\big|\nabla |\nabla u^1|\big|^2\right].
\end{align*}
 By discussion in the proof of Theorem 6 in \cite{la} and Theorem 1.1
in \cite{al2}, we know that equalities in Formulas (1.3a) and (1.3b)
in \cite{al2} hold provided that the level curves are circles (see also Remark~\ref{rem-iso2}). This
observation leads us to two cases: either $\Om'$ is a ball or an
annulus. In the first case, the assumption that $u^1=k$ on $\partial
\Om'$ together with the maximum principle for \p-harmonic functions
(see e.g. \cite[Chapter 6]{hkm}) imply that $u^1\equiv k$ in $\Om'$.
If $\Om'$ is an annulus, then since $u^1$ is constant on two
components of the boundary of $\Om'$, the boundary data is
rotationally invariant. This, together with the uniqueness of
Dirichlet problem for the strictly convex \p-harmonic energy
$\int_{\Om'}|\nabla u^1|^p$ implies that the solution inside $\Om'$
is a radial \p-harmonic function
\begin{equation*}
u^1(r)=c_1 H(r)+c_2, \quad\hbox{where}\quad
H(r)=r^{\tfrac{1}{1-p}}\quad\hbox{and}\quad r=\sqrt{x^2+y^2},
\end{equation*}
whereas constants $c_1$ and $c_2$ depend on the values of $u^1$ on
$\partial \Om'$. Easy computations reveal that for such radial $u^1$
it holds that $L'(s)=-4\pi (p-1)s \ln H'(s)$, $L''(s)=8\pi (p-1)s
\left(\ln H'(s)\right)^2$ and so
\begin{equation}\label{iso-ineq-form}
\frac{(L'(s))^2}{L(s)}=(p-1)L''(s).
\end{equation}
Now, if $p=2$, then the last equation reads:
$ \frac{(L'(s))^2}{L(s)}-L''(s)=0 \Leftrightarrow  (\,\ln L(s)\,)''= 0$, resulting in the equality in (\ref{harmonic-case}).
If $p\not=2$, then (\ref{iso-ineq-form}) equivalently can be written as $\left(\tfrac{p-1}{p-2}L(s)^{\tfrac{p-2}{p-1}}\right)''=0$.
We, therefore, retrieve formula (1.3b) from \cite{al2} and the claim
follows.
\end{proof}

\appendix
\section{Appendix}
\subsection{The matrix $A(f,g)$.}\label{App-system}
 One of the main observations used throughout the paper is that one
can associate with the \p-harmonic system in the plane a quasilinear
system (\ref{system}) and a matrix $A(f,g)$  (see for instance the
work of Vekua \cite{ve} for more applications of such an approach as
well as Section 2 in Alessandrini-Magnanini \cite{alman}). For the
readers convenience we now recall the estimates for the entries of
matrix $A(f,g)$ in formula (\ref{operator}) and necessary
definitions and notation (cf.~\cite{a1} for the complete
discussion). Additionally, we improve the norm estimates for $A(f,
g)$ comparing to Theorem 2 in \cite{a1}, as now we allow
$1<p<\infty$. This extension is due to $C^2$  assumption on mapping
$u$. Indeed, in \cite{a1} we need Lemma~1 in order to infer the
higher regularity of auxiliary expression depending on $Du$,
available only for $p\geq 2$, due to techniques we use. However, if
$u$ is $C^2$, then Lemma~1 in \cite{a1} is no longer needed to
formulate system (\ref{system}). Let us also comment, that in the
setting of planar \p-harmonic functions the similar analysis for a
Sobolev solutions in the full range of parameter $p$ is possible, if
one uses the stream functions \cite{arli} or a variational approach
\cite{3}, both unknown in the vectorial setting.\par

\noindent The following matrix $A$ is introduced for the purpose of
solving system (\ref{system}) in the operator form (\ref{operator}):
\begin{eqnarray}\label{Adef}
A(f,g)\,=\, \frac{2-p}{\Phi}\!\!\left[\begin{array}{cc}
B\frac{\overline{f}}{f}+(2-p)D \phantom{111111111111111} \,\,&
\frac{\overline{g}}{\overline{f}}\big(B\frac{\barf}{f}+(2-p)D\big)
\phantom{1111111111}
\\
\\
\frac{\overline{f}}{Bg}\big(\Phi+(2-p)^2C\big)+(2-p)\frac{\overline{f}}{\overline{g}}D
\quad & \frac{\overline{g}}{Bg}\big(\Phi+(2-p)^2C \big)+(2-p)D
\end{array} \right]
\end{eqnarray}
for
\begin{eqnarray}
\Phi
&:=&\Phi(f,g)=\bigg(2p+\frac{4|g|^2}{|f|^2}\bigg)\bigg(2p+\frac{4|f|^2}{|g|^2}\bigg)
-(2-p)^2\bigg(2+\frac{\overline{g}}{g}\frac{f}{\overline{f}}+\frac{g}{\overline{g}}\frac{\overline{f}}{f}\bigg),
\nonumber \\
B&:=&B(f,g)=2p+4\frac{|f|^2}{|g|^2},\label{Adef-coeff} \\
C&:=&C(f,g)=2+\frac{\overline{g}}{g}\frac{f}{\overline{f}}+\frac{g}{\overline{g}}\frac{\overline{f}}{f},
\nonumber \\
D&:=&D(f,g)=\frac{\overline{g}}{g}+\frac{\overline{f}}{f}. \nonumber
\end{eqnarray}
As in \cite{a1} we introduce
\begin{equation*}
\Phi=
\bigg(2p+\frac{4|g|^2}{|f|^2}\bigg)\bigg(2p+\frac{4|f|^2}{|g|^2}\bigg)
-(2-p)^2\bigg(2+\frac{\overline{g}}{g}\frac{f}{\overline{f}}+\frac{g}{\overline{g}}\frac{\overline{f}}{f}\bigg)
\end{equation*}
and show that
\begin{equation*}
|\Phi|\geq 16p + 8p\bigg(\frac{|f|^2}{|g|^2}+\frac{|g|^2}{|f|^2}\bigg).
\end{equation*}
Then
\begin{align}
|A_{11}(f,g)| & =
\Bigg|\frac{(2-p)\big[\big(2p+\frac{4|f|^2}{|g|^2}\big)\frac{\overline{f}}{f}+(2-p)\big(\frac{\overline{g}}{g}+\frac{\overline{f}}{f}\big)
\big]}{\big(2p+\frac{4|g|^2}{|f|^2}\big)\big(2p+\frac{4|f|^2}{|g|^2}\big)
-(2-p)^2\big(2+\frac{\overline{g}}{g}\frac{f}{\overline{f}}+\frac{g}{\overline{g}}\frac{\overline{f}}{f}\big)}
\Bigg| \nonumber \\
& \leq  \frac{|2-p|\big[2p+\frac{4|f|^2}{|g|^2}+2|2-p|\big]}{16p +
8p\big(\frac{|f|^2}{|g|^2}+\frac{|g|^2}{|f|^2}\big)} \nonumber
\\
& \leq
\frac{2|2-p|(p+|2-p|)|g|^2|f|^2+4|2-p||f|^4}{8p(|f|^2+|g|^2)^2}
\label{A-coeff}
\\
& \leq
(|f|^2+|g|^2)^2\frac{\textrm{max}\{4|2-p|,(p+|2-p|)|2-p|\}}{8p(|f|^2+|g|^2)^2}.\nonumber
\end{align}
Now, we find that
\begin{equation}
A_p:=|A_{11}(f,g)|\leq
\begin{cases}  \label{A-entries}
 \frac{2-p}{2p} & \hbox{for}\quad 1<p<2, \\
 \frac{p-2}{2p} & \hbox{for}\quad 2\leq p\leq3, \\
 \frac{(p-2)(p-1)}{4p} & \hbox{for}\quad 3<p.
\end{cases}
\end{equation}
Similarly we find that the remaining entries $A_{12}(f,g), A_{21}(f,g), A_{22}(f,g)$ satisfy the same estimates in the corresponding ranges of $p$.

\begin{rem}\label{rem_A}
Formulas (10) and (11) in \cite{a1} are slightly different then the
above estimates for $A(f, g)$, but one can show that in fact we have
now improved estimates used in the proof of Theorem 2 in
\cite{a1}.\par
The definition of quasiregular maps in terms of the Beltrami
coefficient (\ref{beltrami}) together with the above estimates
allow us to describe when the complex gradients $f$ and $g$ are
quasiregular (see \cite[Section 3]{a1} for more details). Indeed, as
mentioned in Section~\ref{section2}, system of equations
(\ref{system}) can be solved with the help of matrix $A(f,g)$
resulting in equation (\ref{operator}):
\begin{equation*}
\left[ \begin{array}{c}
 f \\
 g
\end{array} \right]_{\overline{z}} =
A(f,g) \left[ \begin{array}{c}
 f \\
 g
\end{array} \right]_{z}\,+
\overline{ A(f,g)} \overline {\left[ \begin{array}{c}
 f \\
 g
\end{array} \right]_{z}
 }.
\end{equation*}
From this, we have that
\begin{align}
|f_{\overline{z}}| &\leq
|A_{11}(f,g)||f_z|+|A_{12}(f,g)||g_z|+|\overline{A_{11}(f,g)}||\overline{f_z}|+
|\overline{A_{12}(f,g)}||\overline{g_z}| \nonumber \\
& \leq 2A_p\left(|f_z|+|g_z|\right), \label{qr-coeff}
\end{align}
where $A_p$ is as in (\ref{A-entries}). From inequality
(\ref{qr-coeff}) we immediately obtain that
$\frac{|f_{\overline{z}}|}{|f_z|} \leq
2A_p\left(1+\frac{|g_z|}{|f_z|}\right)<1$ provided that
$\frac{|g_z|}{|f_z|}<\frac{1-2A_p}{A_p}$. The similar condition can
be derived for $g$.
\end{rem}
\subsection{Maximum principle for coordinate functions of \p-harmonic maps}\label{App-max}

The purpose of this short section is to show the maximum principle
for the coordinate functions of the \p-harmonic mapping. We use this
principle in part (3) of Remark~\ref{rem-iso2}. To our best
knowledge this result has not appeared in the literature so
far\symbolfootnote[1]{This section is adapted from Section 4.4 in
\cite{phd}. The result holds for all dimensions $n\geq 2$.}. In the
proof below we will use the approach by Leonetti and Siepe, see the
proofs of Theorems~2.1 and~2.2 in \cite{lesi}.
\begin{observ}\label{comp-princ}
 Let $u\in W^{1,p}(\Omega,{\mathbb R}^2)$ be a \p-harmonic mapping in
the domain $\Om\subset \mathbb{R}^2$. If for some $u^i$, $i=1,2$
there exists $k\in {\mathbb R}$ such that $u^i\,\leq\,k$ on
$\partial \Omega$, then $u^i\,\leq\,k$ in $\Om$.
\end{observ}
\noindent Before giving the proof, let us state the following
remark.
\begin{rem}
 Let $v\in W^{1,p}(\Om, \mathbb{R})$. Then the assumption $v\leq l$ on $\partial \Omega$ means that there exists a sequence $\{v_k\}$ of
 a Lipschitz functions on the closure of $\Omega$ such that $v_k(x)\leq l$ for every $x \in \partial \Omega$,
 for each $k\in {\mathbb N}$ and $\|v-v_k\|_{_{W^{1,p}(\Omega,{\mathbb R})}}\to 0, \hbox{ as } k \to \infty$.
\end{rem}
\begin{proof}[Proof of Theorem~\ref{comp-princ}]
Without loss of generality, let us assume that $i=1$. Consider the
following perturbation of mapping $u$:
$$ \tilde{u}\,=\,(u^1+\phi, u^2),$$
where $\phi\,=\,-{\rm max}\{u^1-k, 0\}.$ As $\max\{u^1-k,
0\}\,\in\,\,W_{0}^{1,p}\,(\Omega,\,\R)$ we have that $u$ and
$\tilde{u}$ have the same trace. Define sets
\begin{equation}
\Omega_1\,=\,\{u^1\,\leq\, k\}\cup\{u^1>k,\,\nabla u^1\,=\,0\}, \qquad
\Omega_2\,=\,\Omega\setminus\Omega_1. \nonumber
\end{equation}
Then
\begin{equation}
|D\tilde{u}|^{p}\,=\,|Du|^{p}\quad {\rm on }\,\,\Omega_1\, \hbox{ a.e.}\quad \hbox{ and }\quad
|D\tilde{u}|^{p}\,<\,|Du|^{p}\quad {\rm on }\,\,\Omega_2\, \hbox{ a.e.}
\nonumber
\end{equation}
Uniqueness of the \p-harmonic minimizer implies that $|\Omega_2|\,=\,0$. From this we obtain that $\nabla \phi =0$ a.e. in $\Omega$. Since $\phi \in W^{1,p}_{0}(\Omega, \mathbb{R})$ we have by the Poincar\'e inequality
$$\|\phi\|_{L^p(\Omega)}\,\leq\,c(p)|\Omega|^{\frac{1}{2}}\|\nabla \phi\|_{L^p(\Omega)}\,=0.$$
Thus $\phi\,=\,0$ a.e. in $\Omega$ and the definition of $\phi$ immediately implies that $|\{x\in \Omega\,|\, u^1>k\}|\,=\,0$, completing the proof.
\end{proof}
\subsection{Radial \p-harmonic surfaces and Theorem~\ref{hessian}}\label{radial-hessian}

In Remark~\ref{hessian-rem-rad} we mention the difficulty with
finding wide classes of nontrivial examples when dealing with the
\p-harmonic world. The class of \p-harmonic solutions that comes
most in handy is the one of radial transformations. In
Observation~\ref{obser-rad} we use radial \p-harmonic surfaces to
show that Theorem~\ref{hessian} may fail beyond the range of parameter
$p\in \left<\frac43, 2+\sqrt{2}\right>$. Let
\[
u(x,y)=(u^1, u^2)= (H(r)x, H(r)y),\quad \hbox{for} \quad
r=\sqrt{x^2+y^2}
\]
be a radial map in a planar domain. For such $u$ the \p-harmonic
system (\ref{system_2}) reduces to a single ODE:
\begin{equation}\label{radial-phm}
 (p-1)H''(H')^2r^3+(2p-1)(H')^3r^2+2(p-1)HH'H''r^2 +(5p-4)H(H')^2r + p H^2H''r + 3pH^2H'=0.
\end{equation}
The following formulas hold for $u^1$:
\begin{align}
 \nabla u^1&= \big(H'(r)\frac{x^2}{r}+H(r), H'(r)\frac{xy}{r}\big),\nonumber \\
u^1_{xx}&=H''(r)\frac{x^3}{r^2}+H'(r) \frac{2y^3+3x^2y}{r^3},\quad
u^1_{xy}=H''(r)\frac{x^2y}{r^2}+H'(r) \frac{y^3}{r^3},\quad
u^1_{yy}=H''(r)\frac{xy^2}{r^2}+H'(r)\frac{x^3}{r^3}.\nonumber
\end{align}
Similarly we find $\nabla u^2$ and $u^2_{xx}, u^2_{xy}, u^2_{yy}$. After lengthy computations, we
arrive at equations for Hessian determinants of $u^1$ and $u^2$:
\begin{align}
 \det H(u^1)&=H'(r)H''(r)\frac{x^2}{r}+(H'(r))^2\frac{2x^2-y^2}{r^2}, \nonumber \\
 \det H(u^2)&=H'(r)H''(r)\frac{y^2}{r}+(H'(r))^2\frac{2y^2-x^2}{r^2}. \nonumber
\end{align}
\begin{observ}\label{obser-rad}
 If $p>6+4\sqrt{2}$, then there exist constant $c>1$ and radial \p-harmonic map $u=(u^1, u^2)=H(r)(x,y)$ with $H'\leq 0$ defined in the domain
$\Om\subset \{(x,y)\in \mathbb{R}^2 : c y^2> x^2 > y^2 \}$ such that
 $\det H(u^2)\geq 0$ and $\det H(u^1)\geq 0$.
Thus, Theorem~\ref{hessian} does not hold in general.
\end{observ}
\begin{proof}
Using the above computations for Hessians of $u^1$ and $u^2$, the
proof reduces to finding $u$ with the following properties:
\begin{align}
 \det H(u^2)\geq 0 \Leftrightarrow (H'\geq 0 \hbox{ and } H''r+H'(2-\frac{x^2}{y^2})\geq 0) \hbox{ or } (H'\leq 0 \hbox{ and } H''r+H'(2-\frac{x^2}{y^2})\leq 0), \nonumber \\
 \det H(u^1)\geq 0\Leftrightarrow (H'\leq 0 \hbox{ and } H''r+H'(2-\frac{y^2}{x^2})\leq 0) \hbox{ or } (H'\geq 0 \hbox{ and } H''r+H'(2-\frac{y^2}{x^2})\geq 0). \label{hes-ineq}
\end{align}
For the simplicity of discussion, from now on we will assume that
$H>0$. Such an assumption is justified by the fact that if $u$ is
\p-harmonic, then so is $\tilde{u}=u+(Cx,Cy)$ for a constant $C$.
Thus, by shifting $H$ by a constant we may ensure the positivity of
$H$. Therefore, condition~(\ref{hes-ineq}) together with requirement
that $H'\leq 0$ reduce the hypothesis of observation to showing the
following inequalities:
\begin{equation}\label{rad-cond}
\begin{cases}
 -2-\frac{H''r}{H'} & \leq -\frac{x^2}{y^2} \\
 -2-\frac{H''r}{H'} & \leq -\frac{y^2}{x^2}.
 \end{cases}
\end{equation}
Furthermore, assumption that $c>1$ and definition of $\Om$ allow us
to check only that
\begin{equation*}
-2-\frac{H''r}{H'} \leq -c < -\frac{x^2}{y^2}\quad \left(< -1 \leq
-\frac{y^2}{x^2}\right).
\end{equation*}
From (\ref{radial-phm}) we find that
\begin{equation*}
 -2-\frac{H''r}{H'}=\frac{(H'r+H)^2+(p-2)H(H'r+H)+H^2}{(p-1)(H'r+H)^2+H^2}.
\end{equation*}
Upon defining $t=\frac{H'r+H}{H}$, condition $
-2-\frac{H''r}{H'}\leq -c$ reads:
\begin{equation}\label{rad-quad}
 (1+c(p-1))t^2+(p-2)t +1+c\leq 0.
\end{equation}
Solutions exist provided that $4(1-p)c^2-4pc+p(p-4)\geq 0$. In such
a case $c$ must satisfy
$-\frac{p}{2(p-1)}-\frac{|p-2|}{2(p-1)}\sqrt{p}\leq c
\leq-\frac{p}{2(p-1)}+\frac{|p-2|}{2(p-1)}\sqrt{p}$. Requiring that
$c>0$ gives us condition $p>4$, while $c>1$ holds if
$p>6+4\sqrt{2}$. By solving inequality (\ref{rad-quad}) for $t$ we
may determine conditions for $H$ and $H'$ under which mapping $u$
satisfies (\ref{rad-cond}). Thus, the proof of observation is
completed.
\end{proof}

\vspace{1cm}
\noindent
{\sc Tomasz Adamowicz}\\
{\sc Institute of Mathematics, Polish Academy of Sciences}\\
{\sc Warsaw, 00-656}\\
{\sc Poland}\\
\texttt{T.Adamowicz@impan.pl}

\begin{thebibliography}{99}
\bibitem{am05}{\sc E. Acerbi, G. Mingione}, \emph{Gradient estimates for the $p(x)$-Laplacean system},
J. Reine Angew. Math. {\bf 584} (2005), 117--148.
\bibitem{phd} {\sc T. Adamowicz}, \emph{On the geometry of \p-harmonic mappings}, thesis (Ph.D.)--Syracuse University, 2008.
\bibitem{a1} {\sc T. Adamowicz}, \emph{On \p-harmonic mappings in the
plane}, Nonlinear Anal., {\bf 71} (2009), 502-511.
\bibitem{al} {\sc G. Alessandrini}, \emph{The length of level lines of solutions of elliptic equations in the plane}, Arch. Rational Mech. Anal. {\bf 102} (1988), no. 2, 183--191.
\bibitem{al2}{\sc G. Alessandrini}, \emph{Isoperimetric inequalities for the length of level lines of solutions of quasilinear capacity problems in the plane}, Z. Angew. Math. Phys. {\bf 40} (1989), no. 6, 920--924.
\bibitem{alman} {\sc G. Alessandrini, R. Magnanini},  \emph{The index of isolated critical points and solutions of elliptic equations in the plane}, Ann. Scuola Norm. Sup. Pisa Cl. Sci. (4) {\bf 19} (1992), no. 4, 567--589.
\bibitem{ar} {\sc G. Aronsson}, \emph{Representation of a \p-harmonic function near critical point in the plane}, Manuscripta Math. {\bf 66}, no. 1, (1989), 73-95.
\bibitem{arli} {\sc G. Aronsson, P. Lindqvist}, \emph{On \p-harmonic functions in the plane and their stream functions}, J. Differential Equations
{\bf 74} (1988), no. 1, 157-178.
\bibitem{6} {\sc K. Astala, T. Iwaniec, G. Martin}, \emph{Elliptic PDE's
and Quasiconformal mappings in the plane}, Princeton University
Press, NJ, 2009.
\bibitem{bi} {\sc B. Bojarski, T. Iwaniec}, \emph{\p-harmonic equation and
quasiregular mappings}, Partial Differential Equations, Banach
Center Publications, Vol. {\bf 19}, Warsaw 1987.
\bibitem{bonkhein}{\sc M. Bonk, J. Heinonen}, \emph{Quasiregular mappings and cohomology}, Acta Math. {\bf 186} (2001), no. 2, 219-238.
\bibitem{dpr06} {\sc L. Diening, A. Prohl, M. Ru\v zi\v cka},
  \emph{Semi-implicit Euler scheme for generalized Newtonian Fluids}, SIAM J. Numer. Anal. {\bf 44} (2006), no. 3, 1172-1190.
\bibitem{glm} {\sc S. Granlund, P. Lindqvist, O. Martio}, \emph{Conformally invariant variational integrals}, Trans. Amer. Math. Soc. {\bf 277} (1983), no. 1, 43--73.
\bibitem{hl95}{\sc  R. Hardt, F. Lin}, \emph{Singularities for \p-energy minimizing unit vectorfields on planar domains}, Calc. Var. Partial Differential Equations {\bf 3} (1995), no. 3, 311--341.
\bibitem{hkm} {\sc J. Heinonen, T. Kilpel\"ainen, O. Martio}, \emph{Nonlinear Potential Theory of Degenerate Ellipitic Equations}, Dover Publications, Inc., 2006.
\bibitem{iko11} {\sc T. Iwaniec, L. Kovalev, J. Onninen}, \emph{Diffeomorphic Approximation of Sobolev Homeomorphisms},
Arch. Ration. Mech. Anal.  {\bf 201} (2011), no. 3, 1047--1067.
\bibitem{io11} {\sc T. Iwaniec, J. Onninen}, \emph{$n$-Harmonic Mappings Between Annuli: The Art of Integrating Free Lagrangians},
to appear in Mem. Amer. Math. Soc., 2011.
\bibitem{ka}{\sc B. Kawohl}, \emph{Rearrangements and convexity of level sets in PDE}, Lecture Notes in Mathematics, {\bf 1150} Springer-Verlag, Berlin, 1985.
\bibitem{la} {\sc P. Laurence}, \emph{On the convexity of geometric functionals of level for solutions of certain elliptic partial differential equations}, Z. Angew. Math. Phys. {\bf 40} (1989), no. 2, 258--284.
\bibitem{lv73} {\sc O. Lehto, K. I. Virtanen}, \emph{Quasiconformal mappings in the plane}, Second edition, Springer-Verlag, New York--Heidelberg, 1973.
\bibitem{lesi} {\sc F. Leonetti, F. Siepe}, \emph{Maximum Principle for Vector Minimizers}, Journal of Convex Analysis
{\bf 12} (2005), no. 2, 267-278.
\bibitem{lew3} {\sc John L. Lewis}, \emph{Capacitary functions in convex rings}, Arch. Rational Mech. Anal. {\bf 66} (1977), no. 3, 201--224.
\bibitem{lin1} {\sc P. Lindqvist}, \emph{On \p-harmonic functions in the complex plane and curvature}, Israel J. Math. {\bf 63} (1988), no. 3, 257--269.
\bibitem{long} {\sc M. Longinetti}, \emph{Some isoperimetric inequalities for the level curves of capacity and Green's functions on convex plane domains}, SIAM J. Math. Anal. {\bf 19} (1988), no. 2, 377--389.
\bibitem{moz} {\sc X.-N. Ma, Q. Ou, W. Zhang}, \emph{Gaussian curvature estimates for the convex level sets of \p-harmonic functions}, Comm. Pure Appl. Math. {\bf 63} (2010), no. 7, 935--971.
\bibitem{3} {\sc J. Manfredi}, \emph{\p-harmonic functions in the
plane}, Proc. Amer. Math. Soc., Vol. {\bf 103(2)} (1988), 473-479.
\bibitem{tal} {\sc G. Talenti}, \emph{On functions, whose lines of steepest descent bend proportionally to level lines}, Ann. Scuola Norm. Sup. Pisa Cl. Sci. (4) {\bf 10} (1983), no. 4, 587--605.

\bibitem{4} {\sc K. Uhlenbeck}, \emph{Regularity for a class of non-linear
elliptic systems}, Acta Math. {\bf 138} (1977), 219-240.

\bibitem{wa05}{\sc C. Wang}, \emph{A compactness theorem of $n$-harmonic maps}, Ann. Inst. H. Poincar\'e Anal. Non Lin\'eaire {\bf 22} (2005), no. 4, 509--519.

\bibitem{we08}{\sc S.-W. Wei}, \emph{\p-Harmonic geometry and related topics}, Bull. Transilv. Univ. Bra\c sov Ser. III {\bf 1(50)} (2008), 415--453.

\bibitem{ve} {\sc I. N. Vekua}, \emph{Generalized analytic functions}, Pergamon Press, London--Paris--Frankfurt; Addison--Wesley Publishing Co., Inc., Reading, Mass., 1962.
\end{thebibliography}
\end{document}